%% file: main.tex
\documentclass[VANCOUVER,STIX1COL]{WileyNJD-v2}

\usepackage{xcolor}
\usepackage{color}
\usepackage{graphicx}
\usepackage{hyperref}
\graphicspath{{figures/}}
\usepackage[capitalise]{cleveref}
\usepackage{verbatim}

\newcommand{\bmat}[1]{\begin{bmatrix}#1\end{bmatrix}}
\newcommand{\mb}[1]{\mathbb{#1}}
\newcommand{\mc}[1]{\mathcal{#1}}

\newcommand{\proj}[1]{\B{\Pi}_{#1}}
\renewcommand{\t}{^\top}
\newcommand{\B}[1]{\boldsymbol{#1}}

\newcommand{\rn}{\mathbb{R}^{n}}
\newcommand{\R}{\mathbb{R}}

\newcommand{\rmm}{\mathbb{R}^{m\times m}}
\newcommand{\rmn}{\mathbb{R}^{m\times n}}
\newcommand{\rnn}{\mathbb{R}^{n\times n}}

\newcommand{\normst}[1]{\left\|#1 \right\|_{\B{T} \rightarrow \B{S}}}

\newcommand{\Ah}{\B{\widehat{A}}}
\newcommand{\Oh}{\B{\widehat{\Omega}}}

\newcommand{\Uh}{\B{\widehat{U}}}
\newcommand{\Vh}{\B{\widehat{V}}}
\newcommand{\Sh}{\B{\widehat{\Sigma}}}

\newcommand{\Xh}{\B{\widehat{X}}}
\newcommand{\Yh}{\B{\widehat{Y}}}
\newcommand{\Zh}{\B{\widehat{Z}}}

\newcommand{\rank}{\mathsf{rank}\,}
\newcommand{\range}{\mathsf{range}\,}
\newcommand{\diag}{\mathsf{diag}}

\articletype{Research Article}%

\received{}
\revised{}
\accepted{}

\begin{document}

\title{{Randomized Algorithms} for Generalized Singular Value Decomposition with Application to Sensitivity Analysis} 
\author[1]{Arvind K. Saibaba*}

\author[2]{Joseph Hart}

\author[2]{Bart van Bloemen Waanders}

\authormark{SAIBABA\textsc{et al}}

\address[1]{\orgdiv{Department of Mathematics}, \orgname{North Carolina State University}, \orgaddress{\state{North Carolina}, \country{USA}}}

\address[2]{\orgdiv{Optimization and Uncertainty Quantification}, \orgname{Sandia National Laboratories}, \orgaddress{\state{New Mexico}, \country{USA}}}

\corres{*Arvind K. Saibaba, Department of Mathematics, Box 8205, Raleigh NC, 27695. \email{asaibab@ncsu.edu}}

\abstract[Summary]{The generalized singular value decomposition (GSVD) is a valuable tool that has many applications in computational science. However, computing the GSVD for large-scale problems is challenging. Motivated by applications in hyper-differential sensitivity analysis (HDSA), we propose new randomized algorithms for computing the GSVD which use randomized subspace iteration and weighted QR factorization. Detailed error analysis is given which provides insight into the accuracy of the algorithms and the choice of the algorithmic parameters. We demonstrate the performance of our algorithms on test matrices and a large-scale model problem where HDSA is used to study subsurface flow.}


\maketitle

\section{Introduction}\label{sec:intro}
\input{1intro.tex}

\section{Background and notation} 
\input{2background.tex}

\section{Algorithms}\label{sec:algorithms}
\input{3algorithms.tex}

\section{Analysis}\label{sec:analysis}
\input{4analysis.tex}

\section{Numerical Experiments}\label{sec:numex}

\input{5numerical_results.tex}

\section{Conclusions and Future Directions}

There are three major contributions in this paper. First, we present a new algorithm for computing the truncated GSVD using randomized subspace iterations and discussion numerical and computational issues associated with its implementation. Second, we present probabilistic analysis of the error in the low-rank representation when a standard Gaussian random matrix is used as the initial guess. This analysis provides insight into the parameters of the algorithm---the number of subspace iterations, and the oversampling parameter. Third, we present an application of our algorithm to computing the sensitivity indices in the HDSA framework, and demonstrate through numerical experiments that our approach is both accurate and computationally efficient. Based on these experiments we recommend using \cref{alg:randsvd} with $q=1$ subspace iterations. If a preconditioner is available, and solves involving $\B{T}$ are expensive, the error can be further lowered using the approach in \cref{ssec:precond}.

Computing the truncated GSVD with randomized algorithms facilitates efficient analysis of large linear operators arising in a variety of contexts. This article presented hyper-differential sensitivity analysis where the use of a truncated GSVD has potential to yield a significant reduction in the number of matvecs. Further, the randomized algorithms enables efficient distributed memory parallelism which facilitates computation for large-scale applications.

There are several possible avenues for future exploration. First, the analysis in \cref{sec:analysis} centered on the error in the low-rank approximation. Following the work in \cite{saibaba2019randomized}, one can analyze the accuracy in the singular values and the singular vectors. Second, the algorithms proposed in this paper use standard Gaussian random matrices for the initial guess $\B\Omega$. Several other choices were proposed in this paper but we did not explore them numerically or analytically. It would be interesting to investigate if the dependence of the error condition number $\kappa_2(\B{T})$ can be further weakened, either by using and developing new preconditioners, or using a different strategy. 

In hyper-differential sensitivity analysis, the matrix $\B{A}$ was defined implicitly and matvecs involving $\B{A}$ and its transpose require solving linear systems. If an iterative solver is used, and the iterations are stopped early, then this causes an error in the matvecs. An analysis of the error in the GSVD and developing appropriate stopping criteria would also be interesting to investigate. Similarly, analyzing inexactness from the inconsistency in $\B{B}$ and its transpose is also worth investigating. In addition, future research may explore parallel load balancing associated with computing matvecs in parallel, and the potential to use information from solves in the first stage of the algorithm to precondition solves in the second stage.

\section*{Acknowledgments}
The work of A.K.S. was supported, in part, by the National Science Foundation through the awards DMS-1821149 and DMS-1745654. This paper describes objective technical
results and analysis. Any subjective views or opinions that might be
expressed in the paper do not necessarily represent the views of the
U.S. Department of Energy or the United States Government. Sandia
National Laboratories is a multimission laboratory managed and
operated by National Technology and Engineering Solutions of Sandia
LLC, a wholly owned subsidiary of Honeywell International, Inc., for
the U.S. Department of Energy's National Nuclear Security
Administration under contract DE-NA-0003525. SAND2020-1576 J.

\bibliography{GSVD}

\end{document}

%% file: 1intro.tex
The singular value decomposition (SVD) is perhaps the most important matrix decomposition from a theoretical and numerical perspective. Applications of the SVD includes image processing (image compression and deblurring), statistics (principal component analysis, canonical correlation analysis), model reduction (proper orthogonal decomposition, discrete empirical interpolation method), and machine learning (principal component analysis). 

In this paper, we are concerned with the generalized SVD (GSVD) motivated by an application in hyper-differential sensitivity analysis (HDSA). Two formulations for the GSVD are introduced by van Loan in~\cite{van1976generalizing}. The first formulation~\cite[Theorem 2]{van1976generalizing}, which involves a pair of matrices $\B{A}, \B{B}$, has found application prominently in inverse problems~\cite{hansen2005rank} and bioinformatics~\cite{alter2003generalized,ponnapalli2011higher}. The second formulation is the focus of this article and considers a matrix $\B{A}$ and two matrices,  $\B{S}$ and $\B{T}$, which define inner products on the range and column space of $\B{A}$, respectively. This formulation, which we refer to as the $(\B{S},\B{T})$-GSVD, arises naturally in a variety of applications including: singular value expansions of compact integral operator (such as Fredholm integral equation of the first kind)~\cite[Section 2.4]{hansen2005rank}, weighted inverse problems~\cite{van1976generalizing}, uncertainty quantification~\cite{saibaba2016randomized},  inverse problems~\cite{saibaba2015fast}, model reduction~\cite{drmac2018discrete}, and hyper-differential sensitivity analysis~\cite{hart_vbw_herzog_HDSA}. 

In many of the aforementioned applications of the $(\B{S},\B{T})$-GSVD, the matrix $\B{A}$ arises from discretizing a differential (or integral) equation while $\B{S}$ and $\B{T}$ define inner products in the discretized function space(s). These matrices are typically large as a result of the need for fine spatial meshes in two or three dimensions. The size of the resulting matrices and complexity associated with computing matrix-vector products (matvecs) mandates efficient algorithms to enable analysis on otherwise intractable problems. Hence our interest is to develop GSVD algorithms that are efficient and parallelizable. In this paper we analyze randomized algorithms for GSVD which are motivated by and applied to HDSA~\cite{hart_vbw_herzog_HDSA}, a recently proposed framework for sensitivity analysis of PDE-constrained optimization solutions, as summarized in~\cref{ssec:hdsa}.

One approach to computing the $(\B{S},\B{T})-$GSVD is to use the Cholesky (or any other appropriate factorization) of $\B{S}$ and $\B{T}$ to transform it into a standard SVD problem. However, computing the Cholesky factorization can be a significant computational bottleneck in large-scale application. A central goal in this paper is to avoid such factorizations of $\B{S}$ and $\B{T}$. Instead, our algorithms only rely on matvecs with $\B{S},\B{T}$ or their inverses. This makes our approach matrix-free and hence advantageous for large scale applications. 

\paragraph{Previous work} We briefly review the literature on algorithms for large-scale computation of the GSVD. In previous work~\cite{saibaba2016randomized}, a randomized algorithm was developed for generalized Hermitian eigenvalue problems (GHEP) with application to computing the Karhunen-Lo\'eve decomposition. The GSVD can be formulated as a GHEP in several different ways (a discussion on this can be found in \cref{ssec:ghep}) and this was used in~\cite{hart_vbw_herzog_HDSA} in the context of sensitivity analysis. In~\cite{saibaba2016randomized}, a randomized algorithm for the GSVD was presented but no analysis for this algorithm was shown. In recent work, \cite{xiang2015randomized,vatankhah2018total} a randomized algorithm GSVD algorithm for the $\B{B}$-GSVD was developed. Here too, there has been no analysis of the randomized algorithms. %While Krylov subspace based iterative  methods are popular for computing eigenpairs and for the standard SVD, to our knowledge there are no 

\paragraph{Contributions} This paper makes several contributions for computing large-scale GSVD using randomized algorithms. We give a summary of the main contributions of this paper, while emphasizing the novelty of our approach.

First, in~\cref{sec:algorithms}, we develop new matrix-free randomized algorithms for efficiently computing the truncated GSVD. The algorithms combine  weighted QR factorizations with randomized subspace iteration and improve upon the algorithms developed in~\cite{saibaba2016randomized}. A detailed analysis of the computational cost is provided and compared to our algorithms with related formulations based on the GHEP.

Second, in~\cref{sec:analysis}, we provide detailed probabilistic analysis of the  error in the low-rank decompositions. The analysis is new and sheds light on the choice of the algorithmic parameters so that the trade-off between computational cost and accuracy is clear. {Motivated by the error analysis, we propose a new algorithm that uses a preconditioner to lower the error in the low-rank approximation. }

Third, in \cref{sec:numex}, we show how to efficiently speed up the computations in the HDSA framework using the proposed randomized algorithms. We demonstrate the performance of our algorithms on a large scale model problem in which we seek to control injection wells to meet desired production well fluid pressure profiles, subject to uncertainty in the heterogeneous subsurface media.

%% file: 2background.tex
We begin by overviewing HDSA in~\cref{ssec:hdsa} which motivates the large-scale GSVD problem that we tackle in this paper.  We review the notation and recall fundamental facts about the SVD and GSVD in~\ref{ssec:svd} and the weighted QR decomposition in~\cref{ssec:weighted_qr}.

\subsection{Hyper-differential sensitivity analysis} \label{ssec:hdsa}
HDSA determines the sensitivity of the solution of an optimization problem with respect to fixed parameters. Of particular interest are optimization problems constrained by partial differential equations (PDEs) in the form
\begin{align} 
& \min\limits_{u,z} J(u,z) \label{general_opt} \\
& \text{s.t. } c(u,z,\theta)=0 \nonumber
\end{align}
where $c(u,z,\theta)$ represent a PDE with state $u$ and parameters $\theta$. In this case, the optimization variable $z$ may correspond to a system design or control, or an unknown parameter to be determined in an inverse problem. The parameters $\theta$ may correspond to physical parameters which are uncertain or variable, but out of necessity are fixed to a nominal value in the formulation of \eqref{general_opt}.

HSDA considers the sensitivity of the solution of \eqref{general_opt} to changes in $\theta$. To define such sensitivity, let $z_0$ denote a locally optimal solution of \eqref{general_opt} for a specified nominal estimate $\theta=\theta_0$. Under mild assumptions, see \cite{hart_vbw_herzog_HDSA}, there exists a function $z^\star(\theta)$ which maps parameters $\theta$ in a neighborhood of $\theta_0$ to local minima $z^\star(\theta)$ in a neighborhood of $z_0$. The Fr\'echet derivative of $z^\star$ with respect to $\theta$, which we denote by $\mathcal D z^\star$, is given by
\begin{eqnarray} \label{eqn:dz}
\mathcal D z^\star = \Pi \mathcal K^{-1} \mathcal B,
\end{eqnarray}
where $\Pi$ is a projection operator, $\mathcal K$ is the Karush Kuhn Tucker operator for \eqref{general_opt}, and $\mathcal B$ is the negative Fr\'echet derivative of the gradient of the Lagrangian of \eqref{general_opt} with respect to $\theta$. Computing the action of $\mathcal D z^\star$ requires a large linear system solve, applying $\mathcal K^{-1}$, and each application of the operator $\mathcal K$ requires multiple PDE solves. Hence applying $\mathcal D z^\star$ is computationally intensive and we seek to mitigate the number of operator applies.

In general, $\theta$ and $z$ are elements of infinite dimensional function spaces $\Theta$ and $Z$, respectively, which we assume to be Hilbert spaces. Upon discretization of the parameters with a finite dimensional basis $\{\theta_1,\theta_2,\dots,\theta_n \}$, we define the hyper-differential sensitivity indices
\begin{eqnarray}
\label{sen_indices}
\mathcal S_i = \frac{\vert \vert \mathcal D z^\star \theta_i \vert \vert}{\vert \vert \theta_i \vert \vert} \qquad i=1,2,\dots,n,
\end{eqnarray}
where the norms are computed in $\Theta$ and $Z$. As the dimension of the discretized parameter space $n$ may be large, computing each $S_i$ through operator applies $\mathcal D z^\star \theta_i$, $i=1,2,\dots,n$ is prohibitive. An efficient alternative is to compute the GSVD of $\mathcal D z^\star$ where the inner products are computed in $\Theta$ and $Z$. If $\mathcal D z^\star$ possess low rank structure, which is common in many applications, the indices \eqref{sen_indices} may be efficiently estimated using the leading singular values and vectors of $\mathcal D z^\star$.

In the notation of~\cref{sec:intro} and the remainder of the paper, $\B{A} \in \rmn$ corresponds to the discretization of $\mathcal D z^\star$. In the scope of HDSA, $m$ and $n$ correspond to the dimension of the discretization of $Z$ and $\Theta$, respectively, both of which are typically large (corresponding to discretizations of PDEs). The matrices $\B{S}$ (not to be confused with the sensitivity index $\mathcal S_i$)and $\B{T}$ encode the function space inner products and their dimensions are likewise dependent on discretizations of PDEs. For HDSA, $\B{A}$ will be large, dense, and only accessible through matvecs which require many PDE solves per matvec, hence our motivation for efficient and parallel algorithms. In most cases, $\B{S}$ and $\B{T}$ will be large and sparse. Matvecs with $\B{S}$ and $\B{T}$ are less computationally intensive than with $\B{A}$ (they do not require PDE solves), but factorizations of $\B{S}$ and $\B{T}$ are undesirable because of their size and the loss of sparsity that typically results.

\subsection{SVD and GSVD}  \label{ssec:svd}
Given a positive definite matrix $\B{M} \in \rnn$ and $\B{x}\in \rn$, we define the weighted inner product $\langle\B{x},\B{y}\rangle_{\B{M}}=\B{x}\t \B{My}$ and the associated vector norm \[\|\B{x}\|_{\B{M}} = \sqrt{\B{x}\t \B{M x}} = \|\B{M}^{1/2}\B{x}\|_2 = \|\B{L}_{\B{M}}\t\B{x}\|_2,\]
where $\B{M}^{1/2}$ is the matrix square root and $\B{L}_{\B{M}}$ is the (lower) Cholesky factor of $\B{M}$, i.e., $\B{M} = \B{L}_{\B{M}}\B{L}_{\B{M}}\t$. Let $\B{A} \in \rmn$ and let $\B{S} \in \rmm$ and $\B{T} \in \rnn$ be symmetric positive definite matrices. The induced matrix norm is 
\[ \normst{\B{A}} = \max_{\B{x} \neq \B{0}} \frac{\|\B{Ax}\|_{\B{S}}}{\|\B{x}\|_{\B{T}}} = \|\B{L}_{\B{S}}\t\B{AL}_{\B{T}}^{-\top} \|_2 ,\]
where $\|\cdot\|_2$ is the spectral norm and $\B{L}_{\B{S}}$ and $\B{L}_{\B{T}}$ are the (lower) Cholesky factorizations of $\B{S}$ and $\B{T}$ respectively. In the proofs, it will be convenient to use the alternative relation 
\[ \normst{\B{A}} =\| \B{S}^{1/2} \B{AT}^{-1/2}\|_2.\]
where $\B{S}^{1/2}$ and $\B{T}^{1/2}$ are the matrix square root of $\B{S}$ and $\B{T}$ respectively; however, the Cholesky formulation is preferable from a computational point of view.  We denote by $\kappa_2(\B{T}) = \|\B{T}\|_2 \|\B{T}^{-1}\|_2$, the condition number of inversion in the spectral norm. 

Following the $(\B{S},\B{T})$-GSVD formulation from~\cite{van1976generalizing}, there exist matrices $\B{U} \in \rmm$ that is $\B{S}$-orthogonal, i.e., $\B{U}\t \B{S U} = \B{I}$ and  $\B{V}$ that is $\B{T}$-orthogonal such that 
\[ \B{U}^{-1}\B{AV} = \B\Sigma,\]
where $\B\Sigma \in \rmn$ is a diagonal matrix containing the generalized singular values (in decreasing order)
\[ \sigma_1 \geq \sigma_2 \geq \cdots \geq \sigma_{\min\{m,n\}} \geq 0.\]
Alternatively,  since $\B{V}^{-1} = \B{V}\t \B{T}$, we can write 
\begin{eqnarray}
\label{eqn:gsvd}
 \B{A} = \B{U\Sigma V}\t \B{T}.
 \end{eqnarray}
For a target rank $k \leq \rank(\B{A})$, we can partition the GSVD \eqref{eqn:gsvd} as
\[ \B{A} = \bmat{\B{U}_k & \B{U}_\perp} \bmat{ \B\Sigma_k & \\ & \B\Sigma_\perp} \bmat{\B{V}_k\t \B{T} \\ \B{V}_\perp\t \B{T}},\]
where $\B{U}_k \in \R^{m\times k}$, $\B{V}_k \in \R^{n\times k}$ and $\B{\Sigma}_k \in \R^{k\times k}$. This form will be useful in the error analysis. 

By contrast, consider the standard SVD of $\B{A} = \B{W}\B{S}_{\B{A}}\B{Z}^\top$, where $\B{W} \in \rmm$ and $\B{Z}\in \rnn$ are orthogonal, whose columns respectively contain the left and right singular values, and $\B{S}_{\B{A}}\in \rmn$ is a diagonal matrix containing the singular values (in decreasing order) 
\[s_1 \geq s_2 \geq \cdots \geq s_{\min\{m,n\}} \geq 0. \]

Letting $s_j(\B{A})$ and $\sigma_j(\B{A})$ denote the $j^{th}$ singular value and generalized singular value of $\B{A}$, respectively. We observe that the generalized singular values of $\B{A}$ equal the singular values of $\B{L}_{\B{S}}\t\B{A}\B{L}_{\B{T}}^{-\top}$, i.e., 
\[ \sigma_j(\B{A}) = s_j(\B{L}_{\B{S}}\t\B{A}\B{L}_{\B{T}}^{-\top}), \qquad j=1,\dots,\min\{m,n\}.\]
The generalized left singular vectors of $\B{A}$ can be obtained as $\B{U} = \B{L}_{\B{S}}^{-\top}\B{W}$ and the generalized right singular vectors can be obtained as $\B{V} = \B{L}_{\B{T}}^{-\top}\B{Z}$. It is easy to verify that $\B{U}$ and $\B{V}$ are $\B{S}$- and $\B{T}$-orthogonal, respectively. The details are in \cite[Theorem 3]{van1976generalizing}. Furthermore, the multiplicative singular value inequalities \cite[Equation (7.3.14)]{horn1990matrix} imply
\[  \frac{s_j(\B{A})}{\sqrt{\|\B{S}^{-1}\|_2\|\B{T}\|_2}} \leq \sigma_j(\B{A}) \leq \sqrt{\|\B{S}\|_2\|\B{T}^{-1}\|_2} s_j(\B{A}) \qquad j=1,\dots,\min\{m,n\}.\]
This bound shows that the generalized singular values and the singular values may be substantially different, if $\B{S} \neq \B{I}_m$ and/or $\B{T}\neq \B{I}_n$.

\subsection{QR in a weighted inner product} \label{ssec:weighted_qr}
 An important algorithmic component for the randomized algorithm is an efficient way of computing the thin QR factorization in the weighted inner product. We review the `PreCholQR' algorithm described in~\cite[Algorithm 2]{lowery2014stability}; we call this algorithm Weighted CholQR. Given a positive definite matrix $\B{W}\in \rmm$, this algorithm produces a QR factorization of $\B{Z} \in \rmn$ with $m \geq n$ such that 
\[ \B{Z} = \B{QR} \qquad \B{Q}\t \B{W Q} = \B{I}_n,\]
and $\B{R}\in \rnn$ is an upper triangular matrix. See \cref{alg:cholqr} for the implementation details.

\LinesNumbered
\begin{algorithm}[!ht]
\DontPrintSemicolon
\caption{Weighted CholQR. }
\label{alg:cholqr}
\SetKwInput{Input}{Input}
\SetKwInput{Output}{Output}
	\Input{Matrix $\B{Z}\in \R^{m\times n}$ with $m \geq n$ and positive definite matrix $\B{W} \in \rmm$.}
	\Output{Matrices $\B{Q} \in \rmn$ and $\B{R}\in\rnn$ satisfying $\B{Z} = \B{QR}$ with $\B{Q}\t \B{WQ} = \B{I}$.}
	\BlankLine
	\tcc{Call this function as: $[\B{Q},\B{R}] = \text{CholQR}(\B{Z},\B{W})$.}
Compute thin QR factorization $\B{Q}_{\B{Z}}\B{R}_{\B{Z}}= \B{Z}$\;
Compute $\B{Q}_{\B{W}} = \B{WQ}_{\B{Z}}$
Compute the Cholesky factorization of $\B{Q}_{\B{Z}}\t\B{Q}_{\B{W}}  = \B{R}_{\B{W}}\t \B{R}_{\B{W}}$\;
Form $\B{R} = \B{R}_{\B{W}}\B{R}_{\B{Z}}$ and $\B{Q} = \B{Q}_{\B{Z}}\B{R}_{\B{W}}^{-1}$\;
\end{algorithm}
In addition to computing $\B{Q}$, we may have to also compute $\B{WQ}$; from  \cref{alg:cholqr}, we see that it equals $\B{Q}_{\B{W}}\B{R}_{\B{W}}^{-1}$. This way, we can compute $\B{WQ}$ without expending additional matvecs with $\B{W}$. The computational cost of \cref{alg:cholqr} can be summarized as: $k$ matvecs involving $\B{W}$ and an additional $\mathcal{O}(nk^2)$ floating point operations (flops). In practice, we use a modification of \cref{alg:cholqr}. This modified algorithm first computes a thin-QR factorization of $\B{Z}$ before applying \cref{alg:cholqr}. See~\cite[Section 3]{lowery2014stability} for a discussion on this algorithm. 
 
We will also need to use projection matrices in weighted inner products. Let $\B{A}^\dagger$ denote the Moore-Penrose inverse of $\B{A}$. The orthogonal projector onto the range of $\B{A}$ is denoted as $\proj{\B{A}} = \B{AA}^\dagger$. If $\B{Q} \in \R^{n\times k}$ has $\B{S}$-orthonormal columns, i.e., $\B{Q}\t \B{S Q} = \B{I}_k$, then $\proj{\B{Q}} = \B{QQ}^\dagger = \B{QQ}\t \B{S}$.

%% file: 3algorithms.tex
In \cref{ssec:basicalg} we propose the basic version of the randomized algorithm for GSVD. Next, in \cref{ssec:subspace}, we propose a randomized algorithm for GSVD based on subspace iteration, with special attention to handling the weighted inner products. A discussion on the computational cost of these algorithms is presented in \cref{ssec:compcost} and we end this section with alternative ways of computing the GSVD by formulating it as a generalized hermitian eigenvalue problem (GHEP)~\cref{ssec:ghep}.

\subsection{Outline of the Basic algorithm}\label{ssec:basicalg}

We first give an informal description of the basic version of the algorithm that helps highlight the overall structure. We call this algorithm basic, since it will be a special case of the more general algorithm in \cref{ssec:subspace}. The basic algorithm is comprised of two stages:
\begin{description}
    \item[Stage 1:] determine a subspace to approximate the range of $\B{A}$ by randomized sampling,
    \item[Stage 2:] compute a low rank approximation of $\B{A}$ by projecting on to the subspace determined in Stage 1 and convert into GSVD format.
\end{description}

An optional postprocessing step truncates the decomposition to the desired target rank $k$.

To execute Stage 1, we draw a random matrix $\B\Omega \in \R^{n\times (k+p)}$, where $k$ is the target rank, $p \geq 0$ is an oversampling parameter that can be used to control the accuracy of the low-rank approximation, and $\ell \equiv k +p \leq \min\{m,n\}$. The specific choice of the distribution of $\B\Omega$ will not be discussed at this point; see the end of \cref{ssec:compcost}. The next step is to compute $\B{Y} = \B{ A\Omega}$ and a basis for its range. Specifically, we use the weighted CholQR described in \cref{alg:cholqr} with $\B{W} = \B{S}$. We then have $\B{Y} = \B{QR}$ and $\B{Q}\t \B{S Q} = \B{I}_\ell$. 

In Stage 2, we can obtain a low-rank approximation of the form 
\[ \B{A} \approx \B{QQ}\t\B{SA} = \proj{\B{Q}}\B{A}, \]
where $\proj{\B{Q}} \equiv \B{QQ}\t \B{S}$ is an $\B{S}$-orthogonal projector.
That is, to obtain a low-rank approximation, we project $\B{A}$ onto the range of $\B{Q}$. A few additional steps are then used to convert this low-rank approximation into the GSVD format
\[ \B{A}  \approx \proj{\B{Q}}\B{A} = \Uh \Sh \Vh\t \B{T}.\]
We now present a method for improved accuracy by using randomized subspace iteration in Stage 1. We can view the basic algorithm, outlined above, as a special case of \cref{alg:randsvd} with $q=0$ steps of the subspace iteration.

\subsection{Improved accuracy via subspace iteration}\label{ssec:subspace}

\LinesNumbered
\begin{algorithm}[!ht]
\DontPrintSemicolon
\caption{Randomized subspace iteration with weighted inner products. }
\label{alg:randsubspace}
\SetKwInput{Input}{Input}
\SetKwInput{Output}{Output}
	\Input{Matrices $\B{A} \in \R^{m\times n}$, $\B{S}\in  \rmm$ and $\B{T}\in \rnn$. Random Matrix %\joeynote{Starting guess or random matrix?} 
	$\B\Omega \in \R^{n\times \ell}$ such that $\ell \leq \min\{m,n\}$}
	\Output{Matrices $\B{Q} \in \R^{n\times \ell}$ such that $\B{Q}\t \B{SQ} = I_{\ell}$. }
	\BlankLine
	\tcc{Call as: $[\B{Q}]$ = RandSubspace($\B{A},\B{S},\B{T},\B\Omega,q$).}
Compute $\B{Y}=\B{A\Omega}$.\;
Compute $[\B{Q},\B{R}] = \text{CholQR}(\B{Y},\B{S})$. \;
\For {$j = 1,\dots,q$} {
Update sketch $\B{Y} = \B{A}\t \B{S Q}$. \;
Compute $[\B{Q},\B{R}] = \text{CholQR}(\B{Y},\B{T}^{-1})$. \;
Update sketch $\B{Y} = \B{A}\B{T}^{-1}\B{Q}$. \;
Compute $[\B{Q},\B{R}] = \text{CholQR}(\B{Y},\B{S})$. \;
}\end{algorithm}

{Stage 1 of the basic algorithm may be executed with Lines 1-2 in \cref{alg:randsvd}. Subspace iteration augments them with Lines 3-8 (in \cref{alg:randsvd}) to attain a better projector via the ``sketch'' 
\[ \B{Y} = ( \B{A} \B{T}^{-1}\B{A}\t\B{S})^q\B{A\Omega} ,\]
where $q \geq 0$ is the number of subspace iterations. The rationale behind this sketch is clear if we plug in the generalized SVD $\B{A} = \B{U\Sigma V}\t \B{T}$ to obtain }

\begin{equation}\label{eqn:subspace} \B{Y} = \B{U} (\B\Sigma \B\Sigma\t )^q \B\Sigma\B{V}\t \B{T\Omega}.\end{equation}
Analytical error bounds and numerical evidence shows that employing subspace iterations improves upon the accuracy of the basic algorithm in \cref{ssec:basicalg}. A naive implementation of the subspace iteration as in~\cref{eqn:subspace},  can have poor behavior in the presence of round-off error. This can be addressed by alternating QR factorizations with products involving $\B{A}$ and $\B{A}\t$. We present a version of the subspace iterations in \cref{alg:randsubspace} that accounts for weighted inner products. The randomized subspace iteration produces a matrix $\B{Q} \in \R^{m\times\ell}$ with $\B{S}$-orthonormal columns, which approximates $\range(\B{A})$ and $\B{A} \approx \B{QQ}\t \B{SA}$. This low-rank approximation can be converted into an approximate $(\B{S},\B{T})$-GSVD format by employing Stage 2 of \cref{alg:randsvd}.

\LinesNumbered
\begin{algorithm}[!ht]
\DontPrintSemicolon
\caption{Randomized generalized SVD. }
\label{alg:randsvd}
\SetKwInput{Input}{Input}
\SetKwInput{Output}{Output}
	\Input{Matrix $\B{A} \in \R^{m\times n}$. Target rank $k\leq \rank(\B{A})$, oversampling parameter $p \geq 0$, such that $\ell = k + p \leq \min\{m,n\}$.}
	\Output{Matrices $\Uh,\Sh,\Vh $ such that $\B{A} \approx \Uh\Sh\Vh^\top\B{T}$.}
	\BlankLine
	\tcp{\underline{Stage 1}: Range finder.}
	Draw random matrix $\B\Omega \in \R^{n\times \ell}$. \;
	 Apply subspace iteration $[\B{Q}]$ = RandSubspace($\B{A},\B{S},\B{T},\B\Omega,q$).\;
	\tcp{\underline{Stage 2}: Computing a low-rank factorization.}
Compute $\B{B} = \B{A}\t \B{S Q}$. \;
	Compute the QR factorization $[\B{Q}_{\B{B}},\B{R}_{\B{B}}] = \text{CholQR}(\B{T}^{-1}\B{B}\t,\B{T})$.\;
	Compute the SVD of $\B{R}_B\t = \B{U}_{\B{B}}\Sh \B{V}_{\B{B}}\t$.\;
 	Compute $\Uh = \B{QU}_{\B{B}}$ and $\Vh = \B{Q}_{\B{B}}\B{V}_{\B{B}}$.\;
 	Truncation step (optional): $\Uh = \Uh(:,1:k)$, $\Sh = \Sh(1:k,1:k)$ and $\Vh = \Vh(:,1:k)$.\;
\end{algorithm}
We will see from the analysis in \cref{sec:analysis}, that the basic version of the randomized algorithm is accurate when $\B{T}$ is well conditioned and the generalized singular values $\sigma_j$ with index $j >  k$ are sufficiently small. The number of subspace iterations $q$ involve a trade-off between computational costs and accuracy, and depend on the decay of the generalized singular values the condition number of $\B{T}$.   % We present a new algorithm for GSVD using subspace iteration that can be used when a more accurate low-rank representation is desired. 
\subsection{Computational cost and alternative formulations}\label{ssec:compcost} We now summarize the computational cost of the randomized GSVD using the subspace iteration. Let $q$ denote the number of subspace iterations, and $\ell$ be the number of columns of $\B\Omega$. The cost in Stage 1 is dominated by the steps in \cref{alg:randsubspace}. From \cref{eqn:subspace}, it is clear that we perform $(q+1)\ell$ products with $\B{A}$, and $q\ell$ matvecs with $\B{A}\t$. Additionally, we have to perform $q\ell$ matvecs with $\B{S}$ and $q\ell $ matvecs with $\B{T}^{-1}$ (or solves with $\B{T}$). The weighted QR involving $\B{S}$ requires $\mc{O}(k^2m)$ flops and those involving $\B{T}$ require $\mc{O}(k^2n)$ flops. We now discuss the cost in Stage 2. Step 4 requires $\ell$ matvecs with $\B{A}\t$ and $\B{S}$ respectively; Step 5 requires $\ell$ matvecs with $\B{T}^{-1}$ and $\B{T}$ and an additional cost of $\mc{O}(nk^2)$ flops. Steps 6 and 7 require $\mc{O}(k^2(m+n))$ flops. A summary of this discussion is available in \cref{tab:compcost}. 
\begin{table}[!ht]
    \centering
    \begin{tabular}{c|c|c|c|c|c|c}
    Stage &      $\B{Ax}$ & $\B{A}\t\B{x}$& $\B{Sx}$ & $\B{Tx}$ &  $\B{T}^{-1}\B{x}$ & Other flops \\ \hline
      Range find. &    $(q+1)\ell$ & $q\ell$ & $q\ell$ &$-$ &$q\ell$ &$\mc{O}(qk^2(m+n))$  \\
      Low-rank fact. & $-$ & $\ell$ & $\ell$ & $\ell$ & $\ell$ & $\mc{O}(k ^3 + k^2(m+n))$ 
     \end{tabular}
    \caption{The computational cost of the randomized GSVD algorithm with subspace iteration. Here, $m$ and $n$ refers to the size of $\B{A}$, $k$ is the target rank, $\ell = k + p$, where $p\geq 0$ is the oversampling parameter, and $q \geq 0$ is the number of subspace iterations.}
    \label{tab:compcost}
\end{table}

\paragraph{Should one use $\B{A}$ or $\B{A}\t$?} One can, in principle, apply \cref{alg:randsvd} (appropriately modified using subspace iteration) to $\B{A}\t$ yielding the approximate $(\B{T}^{-1},\B{S}^{-1})$ GSVD 
\[\B{A}\t \approx \Xh \Sh \Yh\t \B{S}^{-1},\]
where $\Xh$ has $\B{T}^{-1}$-orthonormal columns and $\Yh$ has $\B{S}^{-1}$-orthonormal columns. The singular values contained in $\Sh$ approximate the singular values in $\B{T}^{-1/2}\B{A}\t\B{S}^{1/2}$, as desired. To obtain the $(\B{S},\B{T})$-GSVD of $\B{A}$,  we first observe
\[ \B{A} \approx \B{S}^{-1}\Yh\Sh\Xh\t.\]
Then, we make the transformation $\Uh = \B{S}^{-1}\Yh$ (which has $\B{S}$-orthonormal columns) and $\Vh = \B{T}^{-1}\Xh$ (which has $\B{T}$-orthonormal columns), giving the approximate $(\B{S},\B{T})$ GSVD $\B{A} \approx \Uh\Sh\Vh\t\B{T}$. However, it is immediately clear that we have to perform solves with $\B{S}$ (alternatively, matvecs with $\B{S}^{-1}$). If solves with $\B{S}$ is cheaper than solves with $\B{T}$ (which is required by \cref{alg:randsvd}), then it maybe preferable to use this alternative formulation. Another reason to prefer this alternative formulation is if $\B{S}$ has a lower condition number than $\B{T}$; that is, $\kappa_2(\B{S}) < \kappa_2(\B{T})$. The reason is that the error analysis in \cref{thm:prob}, shows an explicit dependence on the condition number $\kappa_2(\B{T})$; the alternative formulation would involve $\kappa_2(\B{S})$ and therefore may have higher accuracy.  

\paragraph{Choice of $\B\Omega$} In \cref{ssec:basicalg,ssec:subspace}, we left the choice of the random matrix $\B\Omega$ unspecified. We briefly comment on the possible choices. A popular choice, we adopt in this paper, is to take $\B\Omega$ to be a standard Gaussian random matrix. That is, with entries are independent and identically distributed (i.i.d) Gaussian random variables with mean $0$ and variance $1$. This choice of $\B\Omega$ ensures that the amount of oversampling needed in practice can be quite modest. For example, following the arguments in \cite[Equation (5.9)]{gu2015subspace}, we can take $p =20$. Other choices are possible, such as subsampled randomized Hadamard/Fourier transform (SRHT/SRFT), Rademacher random matrices, sparse Rademacher random matrices, etc. A discussion of these choice has been provided in \cite[Section 4.6]{halko2011finding} and \cite[Section 3.9]{tropp2017practical}. {In \cref{ssec:precond}, we propose a different approach for constructing the random matrix $\B\Omega$, that makes use of a preconditioner for $\B{T}$.}

\subsection{Computing GSVD using GHEP}\label{ssec:ghep}
There are three alternatives approaches to compute the GSVD by casting it as a GHEP.  In the first approach, we consider the GHEP
$\B{A}\t \B{S Ax} = \lambda  \B{Tx} .$ %Next, we compute the approximate low-rank decomposition 
The second approach considers a similar GHEP, $\B{AT}^{-1}\B{Ax} = \lambda\B{S}^{-1}\B{x}$. In both approaches, the left and right generalized singular vectors can be extracted from the generalized eigenvectors but require additional matvecs with $\B{A}$ or $\B{A}\t$. The third approach considers the Jordan-Wielandt-type augmented matrix 
\begin{eqnarray}
\label{eqn:JW}
\bmat{ & \B{A} \\ \B{A}\t & }\bmat{\B{x}\\ \B{y}} = \lambda \bmat{\B{S}^{-1} \\  & \B{T}} \bmat{\B{x} \\ \B{y} }.
\end{eqnarray}

In our previous work~\cite{saibaba2016randomized}, we proposed randomized algorithms for GHEP. In principle, we can compute an $(\B{S},\B{T})$-GSVD using any of the three formulations described above. However, there are several reasons to use the proposed algorithms in this paper:
\begin{itemize}
\item
Algorithm~\ref{alg:randsvd} is mathematically simpler because it approximates the GSVD directly rather than reformulating it as a GHEP and post-processing the generalized eigenvalues/vectors.

    \item 
The first two formulations, which involve working with $\B{A}\t \B{SA}$ and $\B{AT}^{-1}\B{A}\t$,  can lead to a loss in accuracy, especially when estimating the smallest singular values; see discussion in~\cite[Section 2.1.4]{bjorck2015numerical} and~\cite[Section 3.2]{larsen1998lanczos} for more details. 

\item The cost of the randomized GHEP using the first two formulations is, roughly speaking, twice as expensive as \cref{alg:randsvd} with $q=0$ since it would require $2\ell$ matvecs with $\B{A}$ and $\B{A}\t$; however, the cost is (nearly) comparable to \cref{alg:randsvd} with $q=1$. The results presented in Section~\ref{sec:numex} show that the proposed algorithm is more accurate than the GHEP approach for a comparable computational budget.  %\joeynote{Should we remove this last sentence? The punch line is an accuracy comparison with comparable cost.}

\item The generalized eigenvalues in the third approach come in pairs of positive and negative eigenvalues of equal magnitude. To compute the $k$ largest positive eigenvalues (needed to estimated the GSVD) with a randomized solver requires computing $2k$ eigenvalues. Additionally, while the generalized eigenvectors are orthogonal up to numerical precision, the components ($\B{x}$ and $\B{y}$ in \eqref{eqn:JW}) which determine the left and right generalized singular vectors may not be orthogonal in practice, particularly those corresponding to smaller singular values.

\end{itemize}

%% file: 4analysis.tex
We present error bounds that give insight into the accuracy of the randomized algorithms for computing the GSVD developed in the previous section. The strategy for the analysis is split into two different stages: {\em deterministic} or ``structural'' in which, we make minimal assumptions regarding the distribution of the random matrix $\B\Omega$ (\cref{ssec:struct}), and {\em probabilistic}, in which we specialize the results to specific distributions of $\B\Omega$ (\cref{ssec:prob}). In \cref{ssec:ghepanalysis}, we also provide analysis of the randomized algorithm for the GHEP.

\subsection{Deterministic analysis}\label{ssec:struct}
We assume that the target rank $1 \leq k \leq \mathsf{rank}(\B{A})$ and define the singular value ratio 
\[ \gamma_k \equiv \frac{\sigma_{k+1}}{\sigma_k} .\]
The {\em gap} between the singular values $\sigma_k$ and $\sigma_{k+1}$ is the inverse of the singular value ratio $\gamma_k$. In applications of interest, there may be a large singular value gap which may be exploited to accelerate the convergence of the subspace iteration. Let us turn our attention to the matrix $\B\Omega \in \R^{n\times \ell}$ and define 
\begin{equation}\label{eqn:omega}
    \Oh_1 \equiv \B{V}_k\t \B{T}\B\Omega \qquad \Oh_2 \equiv \B{V}_\perp\t \B{T\Omega}
\end{equation}  
The following is the only assumption we make on the random matrix  $\B\Omega$
\begin{equation}\label{eqn:assump} \rank(\Oh_1) = k . \end{equation}
This assumption ensures that $\Oh_1 \in \R^{k\times \ell}$ has full row-rank and therefore has a right multiplicative inverse, i.e., $\Oh_1\Oh_1^\dagger = \B{I}_k$.

We present two different types of bounds for the analysis of the randomized subspace iteration. The first bound is {\em gap-dependent} and shows explicit dependence on $\gamma_k$, the ratio of the singular values $\sigma_k$ and $\sigma_{k+1}$, whereas the second bound called {\em gap-independent} does not depend on $\gamma_k$.  We also recall some basic properties that will be needed in our analysis. Let $\B{C},\B{D} \in \R^{n\times n}$ be symmetric. The notation $\B{C} \preceq \B{D}$ means $\B{D}-\B{C}$ is positive semidefinite. Let $\B{M},\B{N}$ be two matrices with the same number of rows and let $\range(\B{N}) \subset \range(\B{M})$; then by~\cite[Proposition 8.4]{halko2011finding} $\proj{\B{N}} \preceq \proj{\B{M}}$. Furthermore, 
\begin{equation}\label{eqn:projineq}
\|\proj{\B{N}}\B{A}\|_2 \leq \|\proj{\B{M}}\B{A}\|_2 \qquad \|(\B{I}-\proj{\B{M}})\B{A}\|_2 \leq \|(\B{I}-\proj{\B{N}})\B{A}\|_2 
\end{equation}

\begin{theorem}[Gap-dependent]\label{thm:gapdep} Let $\B\Omega \in \R^{n\times \ell}$ be a standard Gaussian matrix, such that~\eqref{eqn:assump} holds and let $p \geq 0$ and $\ell \leq \min\{m,n\}$. The outputs of \cref{alg:randsubspace} satisfy the following error bound
\[ \normst{\B{A} - \B{QQ}\t \B{S A}}^2  \leq \|\B\Sigma_\perp\|_2^2 + \gamma_k^{4q}\|\B\Sigma_\perp \Oh_2\Oh_1^\dagger\|_2^2. \]
\end{theorem}
\begin{proof}
 There are three main steps in this proof, which heavily relies on the proof technique of \cite[Theorem 9.1]{halko2011finding}. However, we need to pay close attention to the weighted inner products.
 
\paragraph{Step 1. Converting to standard norms} Using the property of the $\normst{\cdot}$ norms
\[ \normst{(\B{I} - \B{QQ}\t \B{S}) \B{A} } = \|\B{S}^{1/2}(\B{I} - \B{QQ}\t \B{S})\B{AT}^{-1/2}\|_{2}. \] Plugging in the generalized SVD of $\B{A}$, we have 
\[ \begin{aligned}
\B{S}^{1/2}(\B{I} - \B{QQ}\t \B{S})\B{AT}^{-1/2} = & \> \B{S}^{1/2}(\B{I} - \B{QQ}\t \B{S})\B{U\Sigma V}\t\B{T}\B{T}^{-1/2} \\
= & \> \B{S}^{1/2}(\B{U} - \B{QQ}\t \B{S U})\B{\Sigma V}\t\B{T}^{1/2}\\
= & \> \B{S}^{1/2} \B{U}(\B{I} - \B{U}\t\B{S}\B{QQ}\t \B{S U})\B{\Sigma V}\t\B{T}^{1/2} & (\text{since  } \B{UU}\t\B{S} = \B{I}) \\
= & \> \B{S}^{1/2} \B{U}(\B{I} - \proj{\B{U}\t\B{SQ}})\B{\Sigma V}\t\B{T}^{1/2}. &
\end{aligned} \]
In the last step, we have used the fact that $\B{U}\t\B{SQ}$ has orthonormal columns.  Combining the intermediate steps we have
\[ \begin{aligned}
\normst{(\B{I} - \B{QQ}\t \B{S}) \B{A} } = & \> \|\B{S}^{1/2}(\B{I} - \B{QQ}\t \B{S})\B{AT}^{-1/2}\|_{2} \\ 
= & \| \B{S}^{1/2} \B{U}(\B{I} - \proj{\B{U}\t\B{SQ}})\B{\Sigma V}\t\B{T}^{1/2} \|_2 = \|(\B{I} - \proj{\B{U}\t \B{S Q}}) \B\Sigma \|_2,
\end{aligned}\]
In the last step, we have used the fact that $\B{S}^{1/2}\B{U}$ and $\B{T}^{1/2}\B{V}$ are orthogonal matrices, and that the spectral norm is unitarily invariant.

\paragraph{Step 2. Reducing dimension from $\ell$ to $k$}
From the generalized SVD of $\B{A}$
 \[ \B{U}\t \B{S Y} = \B{U}\t\B{S} \B{U}(\B\Sigma\B\Sigma\t)^q\B\Sigma\B{V}\t \B{T}\B\Omega = \bmat{\B\Sigma_k^{2q+1} \Oh_1 \\ (\B\Sigma_\perp\B\Sigma_\perp\t)^q\B\Sigma_\perp \Oh_2}.\]
 By assumption, $\Oh_1$ has a right multiplicative inverse and define
 \[ \B{Z} \equiv \B{U}\t \B{S Y}\Oh_1^\dagger \B\Sigma_k^{-(2q+1)} = \bmat{\B{I} \\ \B{F}}, \quad \text{where} \quad\B{F} \equiv (\B\Sigma_\perp\B\Sigma_\perp\t)^q\B\Sigma_\perp \Oh_2\Oh_1^\dagger \B\Sigma_k^{-(2q+1)}.\]
 From the above equation, we have $\range(\B{Z}) \subset \range(\B{U}\t\B{SY}) \subset \range(\B{U}\t\B{SQ})$. Therefore, 
 \[ \proj{\B{Z}} \preceq \proj{\B{U}\t \B{SQ}} \qquad \B{I} - \proj{\B{U}\t \B{SQ}} \preceq \B{I} - \proj{\B{Z}}.\]
 This step has reduced the dimensionality from an $\ell = k + p$ dimensional space to a $k$ dimensional space, since the number of columns of $\B{Z}$ is $k$.
\paragraph{Step 3. Extracting diagonal subblocks}
  Using \cref{eqn:projineq} and properties of the $\|\cdot\|_2 $ norm, we can write 
\[  \normst{\B{A} - \B{QQ}\t \B{S A}}^2  = \|  (\B{I} - \proj{\B{U}\t \B{S Q}})\B\Sigma\|_2^2 \leq \|\B\Sigma\t (\B{I}-\proj{\B{Z}})\B\Sigma\|_2.\]
 Next, using the definition of the spectral projector 
\[ \proj{\B{Z}} = \bmat{\B{I} \\ \B{F}}(\B{F}\t\B{F} + \B{I})^{-1}\bmat{\B{I} & \B{F}\t}, \]
and 
\[ \B\Sigma\t (\B{I}-\proj{\B{Z}})\B\Sigma = \bmat{\B\Sigma_k\t \B{F}_1\B\Sigma_k & * \\ * & \B\Sigma_\perp\t \B{F}_2\B\Sigma_\perp} \]
where $\B{F}_1 = (\B{I}-(\B{I}+\B{F}\t\B{F})^{-1})$, $\B{F}_2 = (\B{I}-\B{F}(\B{I}+\B{F}\t\B{F})^{-1})\B{F}\t)$ and $*$ denote blocks that do not affect the calculations. Applying~\cite[Proposition 3]{halko2011finding}, we obtain 
\[  \|\B\Sigma\t (\B{I}-\proj{\B{Z}})\B\Sigma\|_2 \leq \|\B\Sigma_k\t\B{F}_1\B\Sigma_k\|_2 + \|\B\Sigma_\perp\t\B{F}_2\B\Sigma_\perp\|_2.\]
Following the proof of \cite{halko2011finding}, $\B{F}_1 \preceq \B{F}\t\B{F}$ and $\B{F}_2 \preceq \B{I}$, so that 
\begin{equation}\label{eqn:inter} \|\B\Sigma\t (\B{I}-\proj{\B{Z}})\B\Sigma\|_2 \leq  \|\B\Sigma_k\t\B{F}\t\B{F\Sigma}_k\|_2 + \|\B\Sigma_\perp\t\B\Sigma_\perp\|_2 =  \|\B{F\Sigma}_k\|_2^2 +\|\B\Sigma_\perp\|_2^2  .\end{equation}
With repeated use of the submultiplicativity inequality, we obtain
$\|\B{F\Sigma}_k\|_2 \leq \gamma_k^{2q}\|\B\Sigma_\perp \Oh_2\Oh_1^\dagger\|_2$. Therefore,
\[  \|(\B{I} - \proj{\B{U}\t \B{S Q}}) \B\Sigma \|_2^2  \leq \|\B\Sigma_\perp\|_2^2   + \gamma_k^{4q}\|\B\Sigma_\perp \Oh_2\Oh_1^\dagger\|_2^2.\] 
Combined with the result of step 1, we have the desired result.
\end{proof}

 The following bound quantifies the accuracy of the low-rank approximation and shows but does not explicitly dependend on the singular value ratio $\gamma_k$.

\begin{theorem}[Gap-independent bound]\label{thm:gapindep}
Assume the same setup as of Theorem~\ref{thm:gapdep} and let $q \geq 0$ be the number of subspace iterations. If $\B{Q}$ is the output of \cref{alg:randsubspace}, then 
\[ \normst{\B{A} - \B{QQ}\t  \B{SA} } \leq \left(1 + \|\Oh_2\Oh_1^\dagger\|_2^2 \right)^{1/(4q+2)} \sigma_{k+1}.  \]
\end{theorem}
\begin{proof}
As in the proof of Theorem~\ref{thm:gapdep}, if $\Ah = \B{S}^{1/2}\B{AT}^{-1/2}$
\[ \normst{\B{A} - \B{QQ}\t \B{SA} } = \| (\B{I}-\proj{\B{S}^{1/2}\B{Q}})\Ah\|_2.  \]
We recall~\cite[Proposition 8.6]{halko2011finding} which says if $\proj{}$ is an orthogonal projector, then for $q \geq 0$ 
\[ \|\proj{}\B{M}\|_2\leq \|\proj{} (\B{MM}^\top)^q\B{M}\|_2^{1/(2q+1)}.\]
Applying the above result with $\proj{} = \B{I}-\proj{\B{S}^{1/2}\B{Q}}$ and $\B{M} = \Ah$, we have  
\[
\begin{aligned}
\| (\B{I}-\proj{\B{S}^{1/2}\B{Q}})\Ah\|_2 \leq & \> \| (\B{I} - \proj{\B{S}^{1/2}\B{Q}})(\Ah\Ah\t)^q\Ah\|_2^{1/(2q+1)}  \\ 
= & \> \| (\B{I} - \proj{\B{S}^{1/2}\B{Q}}) \B{S}^{1/2} \B{B T}^{-1/2}\|_2^{1/(2q+1)} \\
= & \> \| \B{S}^{1/2}(\B{I} - \B{QQ}\t\B{S}) \B{BT}^{-1/2}\|_2^{1/(2q+1)}\\
= & \> \normst{ (\B{I} - \proj{\B{Q}})\B{B}}^{1/(2q+1)},
\end{aligned} \] 
where $\B{B} = (\B{AT}^{-1}\B{A}\t\B{S})^q\B{A}$. Using the GSVD of $\B{A}$, we can see that $\B{B}$ has the GSVD 
\[ \B{B} = \B{U} (\B\Sigma\B\Sigma\t)^q\B\Sigma \B{V}\t \B{T}. \] Following the steps of the proof of \cref{thm:gapdep}, we can see that 
\[ \normst{(\B{I} - \proj{\B{Q}})\B{B}} \leq \left(1 + \|\Oh_2\Oh_1^\dagger\|_2^2 \right)^{1/2}\|(\B\Sigma_\perp\B\Sigma_\perp\t)^q\B\Sigma_\perp\|_2 ,\]
where, as before, $\Oh_2 = \B{V}_\perp\t \B{T \Omega}$ and $\Oh_1= \B{V}_k\t \B{T\Omega}$. Since $\|(\B\Sigma_\perp\B\Sigma_\perp\t)^q\B\Sigma_\perp\|_2 = \sigma_{k+1}^{2q+1}$, we therefore have the inequalities 
\[ \normst{\B{A} - \B{QQ}\t \B{SA}} \leq \normst{ (\B{I} - \proj{\B{Q}})\B{B}}^{1/(2q+1)} \leq \left(1 + \|\Oh_2\Oh_1^\dagger\|_2^2 \right)^{1/(4q+2)} \sigma_{k+1}. \] 
\end{proof}

\subsection{Probabilistic analysis}\label{ssec:prob}
Thus far, we have not discussed specific choices for the distribution of the random matrix $\B\Omega \in \R^{n\times \ell}$. In this subsection, we take $\B\Omega$ to a standard Gaussian random matrix and derive probabilistic results on the accuracy of the low-rank decompositions.  We recall some useful facts about the extreme singular values of standard Gaussian random matrices.

\begin{lemma}\label{lem:gauss} Facts about Gaussian random matrices:
\begin{enumerate}
\item Let $\B{G}_1 \in \mb{R}^{m\times n}$ be a standard Gaussian random matrix. Then 
\[ \mb{P}\left\{\|\B{G}_1 \|_2 \geq \sqrt{m} + \sqrt{n} + t\right\} \leq e^{-t^2/2}.\]
    \item Let  $\B{G}_2 \in \mb{R}^{\ell \times k}$ be a standard Gaussian random matrix with $\ell - k  \geq 2$. Then 
\[\mb{P}\left\{ \|\B{G}_2^\dagger\|_2 > t\right\} \leq \sqrt{\frac{1}{2\pi(p+1)}}\left(\frac{e\sqrt{\ell}}{p+1} \right)^{p+1} t^{-(p+1)}.\]
\end{enumerate}
\end{lemma}
\begin{proof} Consider the function $f(\B{G}) = \|\B{G}\|_2$ which has the Lipschitz constant $1$, since by the reverse triangle inequality $|f(\B{X})-f(\B{Y})|\leq \|\B{X}-\B{Y}\|_2 \leq \|\B{X}-\B{Y}\|_F$. By~\cite[Theorem 5.32]{vershynin2012introduction}, $\mb{E}\,f(\B{G}_1) \leq \sqrt{m} + \sqrt{n}$. Therefore, the result of part 1 follows from \cite[Proposition 10.3]{halko2011finding}. The proof of part 2 is an application of \cite[Proposition A.3]{halko2011finding}. 
\end{proof}

\cref{thm:gapdep,thm:gapindep} both identify the term $\|\Oh_2\Oh_1^\dagger\|_2$ appearing in the error bounds. \cref{lemma:inter} (see below) provides a probabilistic bound for this term when the matrix $\B\Omega$ is a standard Gaussian random matrix. Here, we provide an interpretation for this term. To simplify matters take $\B{T} = \B{I}$. Then, by \cite[Table 1]{zhu2013angles} $\|\Oh_2\Oh_1^\dagger\|_2$ is the tangent of the largest canonical angle between the subspaces spanned by the columns of (a) the random matrix $\B\Omega$, and (b) the right singular vectors $\B{V}_k$. Therefore, this term informally represents the degree of alignment between these two subspaces---in the ideal case, both the subspaces are aligned, so that this term is zero, whereas in the worst case, the two subspaces are orthogonal to each other. When $\B{T}\neq \B{I}$, this interpretation has to be modified slightly; the canonical angles are now with respect to the $\langle\cdot,\cdot\rangle_{\B{T}}$ inner products rather than the standard Euclidean inner product~\cite[Theorem 4.2]{knyazev2002principal}.

Before we state and prove \cref{lemma:inter}, define the matrices \begin{equation}\label{eqn:cov}
    \B\Gamma_1 \equiv \B{V}_k\t \B{T}^2 \B{V}_k \in \R^{k\times k}\qquad \B\Gamma_2 \equiv \B{V}_\perp\t \B{T}^2 \B{V}_\perp \in \R^{(n-k)\times (n-k)}.
\end{equation} 
and the constant
\begin{equation}\label{eqn:const}
 C_g \equiv  \frac{e\sqrt{\ell}}{p} \left(\frac{2/\delta}{\sqrt{2\pi(p+1)}}\right)^{1/(p+1)} \left({\sqrt{n-k} + \sqrt{\ell} +\sqrt{2\log\frac{2}{\delta}} } \right).
\end{equation}
Observe that $\B\Gamma_1$ and $\B\Gamma_2$ are positive definite. 

\begin{lemma}\label{lemma:inter}
Consider the notation and assumptions of Theorem~\ref{thm:gapdep}. Let the oversampling parameter satisfy $p \geq 2$ and $\ell = k+p \leq \min\{m,n\}$.
Then with probability at least $1-\delta$ 
\[\| \Oh_2\Oh_1^\dagger \|_2^2 \leq \|\B\Gamma_2\|_2 \|\B\Gamma_1^{-1}\|_2 \> C_g^2 \leq \kappa_2(\B{T}) C_g^2.\]

\end{lemma}

\begin{proof}
Using the submultiplicativity of the spectral norm $\| \Oh_2 \Oh_1^\dagger\|_2 \leq \|\Oh_2\|_2 \| \Oh_1^\dagger\|_2$. We deal with each term separately.

\textbf{Step 1. Bound for $\Oh_2$} First, consider 
$\Oh_2 = \B{V}_\perp\t \B{T}\B\Omega$.
Denote the columns of $\B\Omega = \bmat{\B\omega_1 & \dots & \B\omega_\ell}$. Verify that the $j$-th column of $\Oh_2$ is Gaussian with zero mean and covariance $\B\Gamma_2 \equiv \B{V}_\perp\t \B{T}^2 \B{V}_\perp$, since 
\[ \mb{E}\,[(\B{V}_\perp\t \B{T\omega}_j)((\B{V}_\perp\t \B{T\omega}_j)\t] = \B{V}_\perp\t \B{T} \mb{E}\,[\B\omega_j\B\omega_j\t] \B{TV}_\perp = \B\Gamma_2. \]
It follows that $\B\Gamma_2^{-1/2}\Oh_2 \in \mb{R}^{(n-k)\times\ell}$ is a standard Gaussian random matrix. Applying the first part of \cref{lem:gauss} to $\B\Gamma_2^{-1/2}\Oh_2$, with $ t= \sqrt{2\log\frac{2}{\delta}}$ we get that
\[ \|\Oh_2\|_2 \leq \|\B\Gamma_2^{1/2}\|_2 \|\B\Gamma_2^{-1/2}\Oh_2\|_2 \leq \sqrt{\|\B\Gamma_2\|_2} \left(\sqrt{n-k} + \sqrt{\ell} +\sqrt{2\log\frac{2}{\delta}} \right),  \]
holds with probability of  failure at most $\delta/2$. Note that $\|\B\Gamma_2^{1/2}\|_2 = \sqrt{\|\B\Gamma_2\|_2}$ because $\Gamma_2$ is symmetric positive definite. \\

\textbf{Step 2. Bound for $\Oh_1$} Now, consider $\Oh_1 \in \mb{R}^{k\times \ell }$ whose columns are independent Gaussian vectors   with zero mean and covariance $\B\Gamma_1 \equiv \B{V}_k\t \B{T}^2 \B{V}_k$. So, as before, $\B\Gamma_1^{-1/2}\Oh_1$ is a standard Gaussian random matrix. By \cref{lem:gauss}, 
\[\mb{P}\left\{ \|(\B\Gamma_1^{-1/2}\Oh_1)^\dagger\|_2 > t\right\} \leq \sqrt{\frac{1}{2\pi(p+1)}}\left(\frac{e\sqrt{\ell}}{p+1} \right)^{p+1} t^{-(p+1)}.\]
Set the right hand side to $\delta/2$ and solve for $t$ to obtain

\begin{equation}\label{eqn:oh1}
\mb{P}\left\{ \|(\B\Gamma_1^{-1/2}\Oh_1)^\dagger\|_2 > \left(\frac{2/\delta}{\sqrt{2\pi(p+1)}}\right)^{1/(p+1)}\frac{e\sqrt{\ell}}{p+1}\right\} \leq \frac{\delta}{2}.
\end{equation}
We now simplify $\|(\B\Gamma_1^{-1/2}\Oh_1)^\dagger\|_2$. By~\cite[Theorem 2.2.3]{bjorck2015numerical} if $\B{C}\in \R^{m\times r}$ and $\B{D} \in \R^{r\times n}$ and $\rank(\B{C}) = \rank(\B{D}) = r$, then
\[ (\B{CD})^\dagger = \B{D}^\dagger\B{C}^\dagger\]
Pick $\B{C} = \B\Gamma_1^{-1/2}$ and $\B{D} = \Oh_1$. The above result applies since  $\rank(\B\Gamma_1^{-1/2}) = \rank(\Oh_1) = k$. Therefore, 
\[(\B\Gamma_1^{-1/2}\Oh_1)^\dagger = \Oh_1^\dagger \B\Gamma_1^{1/2},\]
and $ \|\Oh_1^\dagger\|_2 \leq \|(\B\Gamma_1^{-1/2}\Oh_1)^\dagger\|_2 \|\B\Gamma_1^{-1/2}\|_2$.
Therefore,  combining this result with \cref{eqn:oh1}, we get with probability of failure at most $\delta/2$
\[ \|\Oh_1^\dagger\|_2\leq \sqrt{\| \B\Gamma_1^{-1}\|_2} \left(\frac{2/\delta}{\sqrt{2\pi(p+1)}}\right)^{1/(p+1)}\frac{e\sqrt{\ell}}{p+1}.\]

\textbf{Step 3. First bound} Combining the results of steps 1 and 2 and using a union bound, we have
\[ \mb{P}\left\{\|\Oh_2\|_2 \| \Oh_1^\dagger\|_2 > \sqrt{\|\B\Gamma_2\|_2\|\B\Gamma_1^{-1}\|_2} C_g  \right\} \leq \delta/2 + \delta/2  =\delta, \]
where $C_g$ is defined in \cref{eqn:const}.

\textbf{Step 4. Second bound} Partition $\B{V}\t \B{T}^2 \B{V}$ as 
\[ \B{V}\t \B{T}^2 \B{V}  = \bmat{\B{V}_k\t \B{T}^2 \B{V}_k & *\\ *& \B{V}_\perp\t \B{T}^2 \B{V}_\perp},\]
where $*$ denote terms that are unimportant in the calculation. 
Define $\Vh \equiv \B{T}^{1/2}\B{V}$ and notice that  $\Vh$ is orthogonal. This implies $\B{V}\t \B{T}^2 \B{V} = \Vh\t \B{T} \Vh$ and that, by similarity, $\B{V}\t \B{T}^2\B{V}$ and $\B{T}$  have the same eigenvalues. Let $\B{M}\in\rnn$ and let $\lambda_k(\cdot)$ denote  the eigenvalues of a matrix arranged in descending order for $k=1,\dots,n$.  By Cauchy interlacing theorem~\cite[Theorem 3.2.9]{bjorck2015numerical}, 
\[  \lambda_k(\B{T}) = \lambda_k(\B{V}\t \B{T}^2  \B{V}) \geq \lambda_k(\B{V}_k\t \B{T}^2\B{V}_k) \geq \lambda_n(\B{V}\t \B{T}^2 \B{V}) = \lambda_n(\B{T}). \]
Therefore, $\lambda_k^{-1}(\B{T}) \leq  \lambda_k^{-1}(\B{V}_k\t \B{T}^2 \B{V}_k) \leq \lambda_n^{-1}(\B{T})$. Similarly, using Cauchy interlacing theorem
 $$\lambda_1(\B{T}) = \lambda_1(\B{V}\t\B{T}^2\B{V}) \geq \lambda_1(\B{V}_\perp\t \B{T}^2 \B{V}_\perp) \geq \lambda_{k+1}(\B{V}\t\B{T}^2\B{V}) =\lambda_{k+1}(\B{T}).$$ 
 Since $\B\Gamma_1$ and $\B\Gamma_2$ are positive definite, their singular values are their respective eigenvalues; therefore,
 \[ \frac{\lambda_{k+1}(\B{T})}{\lambda_k(\B{T})}\leq \|\B\Gamma_2\|_2 \|\B\Gamma_1^{-1}\|_2 = { \frac{\lambda_{1}(\B{V}_\perp\t \B{T}^2 \B{V}_\perp)}{\lambda_{k}(\B{V}_k\t \B{T}^2\B{V}_k)}}
 \leq \frac{\lambda_1(\B{T})}{\lambda_n(\B{T})} =  \kappa_2(\B{T}).\]
 Combine with the result of step 3, we obtain the second bound
\end{proof}
The proof identifies that the condition number $\kappa_2(\B{T})$ plays a role in the error analysis, which suggests that a large condition number can result in a large error. However, this may be pessimistic for the following reason. The proof also gives a lower bound $\lambda_{k+1}(\B{T})/\lambda_k(\B{T})$ which can be attained by a specific instance of $\B{V}$. This lower bound shows that an ill-conditioned matrix $\B{T}$ does not necessarily amplify the error in the low-rank approximation.

We can prove the following probabilistic bound for the error in the low-rank approximation.
\begin{theorem}\label{thm:prob}
Let $\B\Omega \in \R^{n\times \ell}$ be a standard Gaussian random matrix where $k\leq \rank(\B{A})$ is the target rank. Let the oversampling parameter $p$ satisfy $p \geq 2$, $\ell = k + p  \leq \min\{m,n\}$,  and the failure rate  $0 < \delta <1 $. Let $\B{Q}$ be computed using \cref{alg:randsubspace} with $q \geq 0$ subspace iterations. With probability at least $1-\delta$ 
\[\normst{\B{A} - \B{QQ}\t \B{S A}} \leq\left(1 + \gamma_k^{4q+2}\kappa_2(\B{T})C_g^2 \right)^{1/2} \sigma_{k+1},\]
and 
\[ \normst{\B{A}- \B{QQ}\t \B{S A}} \leq \left(1 + \kappa_2(\B{T})C_g^2 \right)^{1/(4q+2)} \sigma_{k+1} ,\]
where the constant $C_g$ is defined in \eqref{eqn:const}. 
\end{theorem}

\begin{proof}%[Proof of Theorem~\ref{thm:prob}]

Plug the results of \cref{lemma:inter} into the statement of \cref{thm:gapindep,thm:gapdep} give the desired results.
\end{proof}
We make several remarks. First, if the matrix $\B{A}$ has rank$-k$, then with high probability we recover the exact rank-$k$ generalized SVD since $\sigma_{k+1} = 0$. Second, the bounds depend on the condition number of the weighting matrix $\kappa_2(\B{T})$. Third, the value of the constant $C_g$ is likely to be not optimal and it may be possible to improve upon. It is easily seen that as $q\rightarrow \infty$, the factor $\left(1 + \kappa_2(\B{T})C_g^2\right)^{1/(4q+2)}$ approaches $1$ and $$\normst{\B{A}-\B{QQ}\t \B{SA}} \approx \sigma_{k+1}.$$ Similarly, if $\gamma_k < 1$, then $\gamma_k^{4q}$ approaches $0$, so that once again the above approximation holds.

{
 \subsection{Using a preconditioner for improved accuracy}\label{ssec:precond}
 From the error analysis in \cref{thm:gapdep,thm:gapindep}, we see that the error bounds have the factor of $\kappa_2(\B{T})$, which can be large if $\B{T}$ is ill-conditioned. Motivated by this observation, we propose a new choice of the distribution of $\B\Omega$ to mitigate the issue due to ill-conditioning.

 Assume that we have an approximate factorization of the form $\B{T}^{-1} \approx \B{LL}\t$, so that we call $\B{L}$ the preconditioner of $\B{T}$.  We have two requirements of the preconditioner: (1) the condition number $\kappa_2(\B{L}\t\B{TL})$ should be small compared to the condition number $\kappa_2(\B{T})$, and (2) the operation $\B{L}^{-1}\B{x}$ should be cheap to perform. For example, such a factorization maybe available using the Incomplete Cholesky factorization or the Sparse Approximate Inverse preconditioner (SPAI) approach \cite{saad2003iterative}.

 If such a preconditioner is available, we show how to improve the accuracy of the randomized algorithm for GSVD. The main idea is to sample from a different distribution to construct the matrix $\B\Omega$. Let $\{\B\omega_j\}_{j=1}^\ell$ be independent samples drawn from  $\mc{N}(\B{0},\B{LL}\t)$ and  let 
 \[\B\Omega = \bmat{\B\omega_1 & \dots &  \B\omega_\ell} \in \R^{n\times \ell}.\]
  Note that we can write $\B\Omega = \B{LG}$, where $\B{G}\in \R^{n\times \ell}$ is a standard Gaussian random matrix. To see this, let $\B{g}_j$ be the j-th column of $\B{G}$. Then, $\mb{E}[\B\omega_j] = \B{0}$ and
  \[ \mb{E}\,[ \B{\omega}_j\B\omega_j\t] = \B{L}\mb{E}\,[ \B{g}_j\B{g}_j\t]\B{L}\t= \B{LL}\t    .\]
  We invoke \cref{alg:randsvd} with the matrix $\B\Omega$ constructed as before, which we call the preconditioned Gaussian random matrix. The only addition to the computational cost in \cref{alg:randsvd} is $\ell$ additional solves involving $\B{L}$. The following result captures the error in the low-rank approximation.
  \begin{theorem}[Preconditioned Gaussian random matrix]\label{thm:prob_prec}
Let $\B\Omega = \B{LG} \in \R^{n\times \ell}$, where $\B{G}$ is a standard Gaussian random matrix. Consider the same assumptions and notation as of \cref{thm:prob}. With probability at least $1-\delta$ 
\[\normst{\B{A} - \B{QQ}\t \B{S A}} \leq\left(1 + \gamma_k^{4q+2}\kappa_2(\B{L}\t\B{TL})C_g^2 \right)^{1/2} \sigma_{k+1}, \]
and
\[ \normst{\B{A}- \B{QQ}\t \B{S A}} \leq \left(1 + \kappa_2(\B{L}\t\B{TL})C_g^2 \right)^{1/(4q+2)} \sigma_{k+1} ,\]
where the constant $C_g$ is defined in \cref{eqn:const}.

\end{theorem}

\begin{proof} From  \cref{thm:gapdep,thm:gapindep}, it is clear that we have to focus on $\|\Oh_2\Oh_1^\dagger\|_2$ which appears in the statement of both theorems. Now define $\B\Upsilon \equiv \B{TLL}\t\B{T}$ and in analogy with \eqref{eqn:cov} define
\[ \B\Gamma_2 \equiv \B{V}_\perp\t \B\Upsilon\B{V}_\perp \qquad \B\Gamma_1 \equiv \B{V}_k\t \B\Upsilon\B{V}_k .\]
Verify that $\B\Gamma_2^{-1/2}\Oh_2 \in \R^{(n-k)\times \ell}$ and $\B\Gamma_1^{-1/2}\Oh_1 \in \R^{k\times \ell}$ are both standard Gaussian random matrices. Therefore, combining steps 1 to 3 of the proof of \cref{lemma:inter}, with probability at least $1-\delta$ 
\[ \|\Oh_2\Oh_1^\dagger\|_2 \leq \|\B\Gamma_2\|_2 \|\B\Gamma_1^{-1}\|C_g^2.\]

 Note that $\B{V}\t \B{\Upsilon} \B{V} = \B{V}\t\B{T}^{1/2}\B{T}^{1/2}(\B{LL}\t)\B{T}^{1/2} \B{T}^{1/2}\B{V} $. Since $\B{T}^{1/2}\B{V}$ is orthogonal,  $\B{V}\t \B{\Upsilon} \B{V}$ and $\B{T}^{1/2}(\B{LL}\t)\B{T}^{1/2}$ have the same eigenvalues. By Cauchy interlacing theorem, we have
 \[ \|\B\Gamma_2\|_2 \|\B\Gamma_1^{-1}\|_2 = { \frac{\lambda_{1}(\B{V}_\perp\t \B{\Upsilon} \B{V}_\perp)}{\lambda_{k}(\B{V}_k\t \B\Upsilon \B{V}_k)}}
 \leq \frac{\lambda_1(\B{T}^{1/2}(\B{LL}\t)\B{T}^{1/2})}{\lambda_n(\B{T}^{1/2}(\B{LL}\top)\B{T}^{1/2})} .\]
If $\B{C},\B{D} \in \rnn$ then $\B{CD}$ and $\B{DC}$ have the same eigenvalues; see~\cite[Theorem 1.3.22]{horn1990matrix}; so, $\B{T}^{1/2}(\B{LL}\t)\B{T}^{1/2}$ has the same eigenvalues as $\B{L}\t\B{TL}$. Therefore, with probability at least $1-\delta$
 \[ \|\Oh_2\Oh_1^\dagger\|_2^2 \leq \|\B\Gamma_2\|_2 \|\B\Gamma_1^{-1}\|_2 C_g^2 \leq \kappa_2(\B{L}\t\B{TL})C_g^2.\]
 Plug this bound into the results of \cref{thm:gapdep,thm:gapindep} to obtain the desired result.
\end{proof}
If the preconditioned operator $\B{L}\t\B{TL}$ has a lower condition number than $\B{T}$, then the bound in \cref{thm:prob_prec} suggests that the error should be lower (this is confirmed by numerical experiments). Of course, from an accuracy perspective, in the ideal case $\B{L} = \B{T}^{-1/2}$ so that the preconditioned operator has a condition number of $1$.

\subsection{Analysis of randomized algorithm for GHEP} \label{ssec:ghepanalysis}
Consider the GHEP
 \begin{equation} \label{eqn:ghep} \B{Ax}=\lambda\B{Bx},
 \end{equation}
 where $\B{A},\B{B}\in \rnn$ are symmetric, $\B{A}$ is positive semidefinite, and $\B{B}$ is positive definite. We can extend the analysis developed in this section to a randomized algorithm for GHEP.

 To this end, consider the generalized eigenvalues of \cref{eqn:ghep} in descending order 
 \[\lambda_1 \geq \dots \geq \lambda_k \geq \lambda_{k+1} \geq \dots \geq \lambda_n.\]
 Let us define $\B{C} \equiv \B{B}^{-1}\B{A}$; note that $\lambda_j$'s are also the eigenvalues of $\B{C}$. Furthermore observe that the $(\B{B},\B{B})$-generalized singular values of $\B{C}$ satisfy for $j=1,\dots,n$ 
    \[\sigma_{j}(\B{C}) = s_j(\B{B}^{1/2}\B{CB}^{-1/2}) =s_{j}(\B{B}^{-1/2}\B{AB}^{-1/2}) = \lambda_{j}.\]

 As before, we draw a Gaussian random matrix $\B\Omega \in \R^{n\times \ell}$ and form $\B{Y} = \B{C\Omega}$. Then, we compute the weighted QR factorization of $\B{Y}=\B{QR}$ where $\B{Q}$ has $\B{B}$-orthonormal columns. Then, we have the low-rank approximation
 \[\B{C} \approx \B{QQ}\t\B{BC} = \proj{\B{Q}}\B{C}. \]
 To analyze the error in this low-rank representation, we can apply \cref{thm:prob} with $q=0$, $m=n$  and $\B{S} = \B{T} = \B{B}$. As before, let $0<  \delta < 1$ be a user-defined parameter that denotes the probability of failure. With probability at least $1-\delta$, 
 \begin{equation}\label{eqn:gheperr} \|(\B{I}-\B{QQ}\t\B{B})\B{C}\|_{\B{B}\rightarrow \B{B}} \leq \left(1 + \kappa_2(\B{B})C_g^2 \right)^{1/2} \lambda_{k+1},\end{equation}
 
 This bound is easy to interpret, since the absolute error is expressed in terms of the $(k+1)$-th generalized eigenvalue. This provides an alternative to \cite[Theorem 1]{saibaba2016randomized}; however, a direct comparison between these two results is difficult since the results are expressed using the $(\B{B},\B{I})-$GSVD of $\B{C}$.  We would also like to mention that the final result in \cite[Theorem 1]{saibaba2016randomized} is missing a factor of $2$ on the right hand side. If a symmetric low-rank representation is desired, we can use the approximation
 \[ \B{C} \approx \proj{\B{Q}}\B{C}\proj{\B{Q}} \qquad\text{or}\qquad  \B{A} \approx \B{QQ}\t\B{BAQQ}\t\B{B},\]
 where $\proj{\B{Q}} = \B{QQ}\t\B{B}$. Using~\cite[Equation (7)]{saibaba2016randomized}, 
 \[  \|\B{C} -\proj{\B{Q}}\B{C}\proj{\B{Q}}\|_{\B{B}\rightarrow \B{B}} \leq 2\|\B{C}-\proj{\B{Q}}\B{C}\|_{\B{B}\rightarrow \B{B}}.\]
 Combined with~\cref{eqn:gheperr}, we have with probability at least $1-\delta$ 
 \[ \|\B{C} -\proj{\B{Q}}\B{C}\proj{\B{Q}}\|_{\B{B}\rightarrow \B{B}} \leq 2\left(1 + \kappa_2(\B{B})C_g^2 \right)^{1/2} \lambda_{k+1}.\]
 This analysis is beneficial for the computation of the GSVD which has been formulated as an appropriate GHEP as in \cref{ssec:ghep}. Furthermore, we note that the basic version of the algorithm can be improved by using a variation of the randomized subspace iteration, \cref{alg:randsubspace}. Additional theoretical results can be derived using the approach in \cref{thm:gapdep} or \cref{thm:gapindep}.

%% file: 5numerical_results.tex
In~\cref{ssec:testexp}, we apply the proposed algorithms on a set of test matrices. The numerical experiments explore various features of the algorithms such as the number of subspace iterations and condition number of the weighting matrices. Furthermore, we compare the performance of the proposed algorithms against previously proposed algorithms. In~\cref{ssub:hdsa_inexact}, we present a numerical study of the robustness of randomized algorithms to inexactness arising in HDSA applications. Then, in~\cref{ssub:hdsainv}, we apply the randomized algorithms to compute hyper-differential sensitivity indices, and demonstrate the accuracy and computational benefits of the proposed algorithms.

\subsection{{Experiments with Test Matrices}}\label{ssec:testexp}
\subsubsection{Description of the test matrices}\label{sssec:test}
For the matrix $\B{A} \in \R^{128\times 128}$, we have  the following choices:
\begin{enumerate}
    \item \textbf{Controlled gap} The first test matrix $\B{A} \in \R^{128 \times 128}$ is constructed using the formula
\[ \B{A} = \sum_{j=1}^{r}\frac{\text{gap}}{j} \, \B{x}_j \B{y}_j^\top +
  \sum_{j=r+1}^{128}\frac{1}{j} \,\B{x}_j \B{y}_j^\top,\]
where $\B{x}_j \in \R^{128}$ and $\B{y}_j \in \R^{128}$ are sparse random vectors with non-negative entries generated using the MATLAB commands \verb|sprand(128,1,0.025)| and \verb|sprand(128,1,0.025)| respectively and $\text{gap} = 10$.
\item \textbf{Low-rank plus noise} which takes  the form
\[ \B{A} = \begin{bmatrix} \B{I}_r & \B{0} \\ \B{0} & \B{0}\end{bmatrix} + \sqrt{\frac{\gamma_\text{noise} r}{2n^2}}(\B{G} + \B{G}\t),\]
where $\B{G} \in \R^{n\times n}$ is a random Gaussian matrix and we take $\gamma_\text{noise} = 10^{-2}$.
\item \textbf{Low-rank plus decay} which takes the form
\[ \B{A} = \diag(\underbrace{1,1,\dots,1}_r,2^{-d},3^{-d}\dots,(n-r+1)^{-d}),\]
where $n=128$ and $d= 1$.
\item \textbf{Decay} which takes the form 
\[ \B{A} = \diag(0.9^1,\dots,0.9^{128}).\]

\end{enumerate}
The test matrices described here, model different scenarios of singular value decay. A more detailed description of these matrices is described in \cite[Section 6.1]{saibaba2019randomized}. We take the parameter $r=15$. The matrices $\B{S}$ and $\B{T}$ are both generate using \textsc{MATLAB}'s \verb|gallery| function. The matrix $\B{S}$ is computed using \verb|gallery('minij',128)|, whereas $\B{T}$ is generated using \verb|gallery(`randsvd', n, -10^4, 5)| with the condition number $10^4$. 

\begin{figure}[!ht]
    \centering
    \includegraphics[scale=0.5]{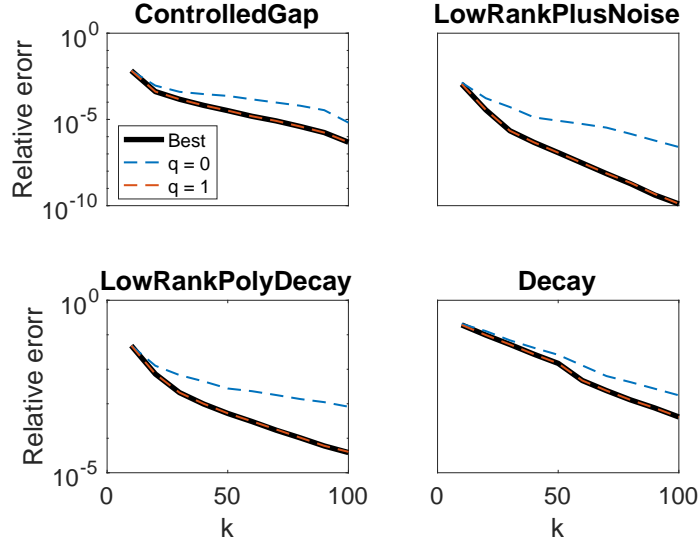}
    \caption{The error in the low-rank approximations obtained using subspace iterations $q=0,1$. `Best' refers to $\sigma_{k+1}/\sigma_1$, the error in the  best rank-$k$ representation.}
    \label{fig:testmatrices}
\end{figure}

For each test matrix, we apply  \cref{alg:randsvd}, and compute a low-rank approximation as a function of the target rank $k$. The error in the low-rank approximation is plotted in \cref{fig:testmatrices}.  The  
oversampling parameter is taken to be $p=10$ and the number of subspace iterations $q=0,1$.  If $\Ah$ is an approximation of $\B{A}$, the  relative error is defined as 
\[ \text{error} \equiv \frac{\normst{\B{A}-\Ah}}{\normst{\B{A}}}. \]
For all the test matrices, we see that the error in the low-rank approximation decreases as the target rank $k$ increases. We also observe that, with increasing subspace iterations, the error in the low-rank approximation decreases and with one subspace iteration $q=1$, the error is comparable to the ``best'' possible error $\sigma_{k+1}/\sigma_1$.

\begin{figure}[!ht]
    \centering
    \includegraphics[scale=0.5]{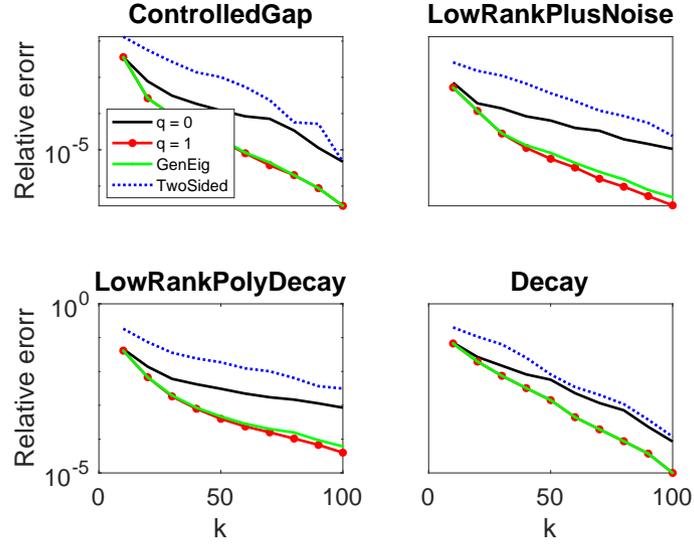}
    \caption{Comparing the accuracy between different methods. 
    }
    \label{fig:comp}
\end{figure}

\subsubsection{Comparing different methods}\label{sssec:comp}
We want to compare the performance of \cref{alg:randsvd}, which we call `GSVD(q)' where $q$ denotes the number of subspace iteration with two other approaches. The first method computes the GSVD by formulating it as a GHEP $\B{A}\t\B{SAx} = \lambda\B{Tx}$ as described in \cref{ssec:ghep}; we call this `GenEig.' The second method, first described in \cite[Section 3]{saibaba2016randomized}, uses a two-sided approach as follows. We draw two random matrices $\B\Omega \in \R^{n\times \ell}$ and $\B\Psi \in \R^{m\times \ell}$ and compute $\B{Y}_{\B\Omega} = \B{A\Omega}$ and $\B{Y}_{\B\Psi} = \B{A}\t\B{\Psi}$. We compute the thin QR factorization of $\B{Y}_{\B\Omega}$ and $\B{Y}_{\B\Psi}$ to obtain $\B{Q}$ (which has $\B{S}$-orthonormal columns), and $\B{Z}$ (which has $\B{T}^{-1}$-orthonormal columns) respectively. This results in the low-rank approximation
\[ \B{A} \approx \B{QQ}\t\B{SAT}^{-1}\B{ZZ}\t.\]
To convert into the GSVD format, we compute $\B{F} = \B{Q}\t\B{SAT}^{-1}\B{Z}$ and its thin SVD $\B{F} = \B{U}_{\B{F}}\Sh \B{V}_{\B{F}}\t$.  We obtain $\B{A} \approx \Uh\Sh\Vh\t\B{T}$ by computing $\Uh = \B{QU}_{\B{F}}$ and $\Vh = \B{T}^{-1}\B{ZV}_{\B{F}}$ which is denoted as the `Two-Sided' approach.  We point out that because of a mathematical error, \cite[Section 3]{saibaba2016randomized} we do not compute the $(\B{S},\B{T})$-GSVD of $\B{A}$, and the procedure described above is used instead. \cref{fig:comp} compares the error in all four test matrices described. The matrices $\B{S}$ and $\B{T}$ are defined as before. We observe that `GSVD(1)' is the most accurate. However,  `GenEig' are computational comparable but slightly less accurate. The `TwoSided' approach is more expensive and has the worst error.  

\begin{figure}[!ht]
    \centering
    \includegraphics[scale=0.5]{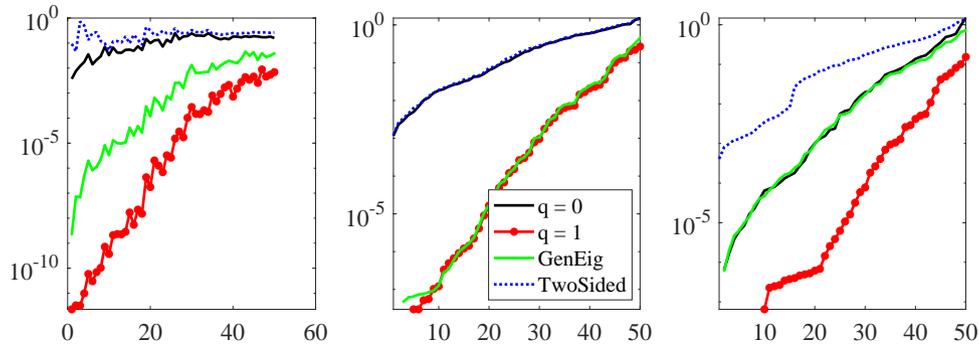}
    \caption{Comparing the accuracy between different methods (left) absolute error in the singular values, (center) canonical angles between the ``exact'' and approximate left singular vectors, and (right) canonical angles corresponding to the right singular vectors.    }
    \label{fig:comp_svs_angles}
\end{figure}
{
\cref{fig:comp} compares the accuracy in the low-rank approximations. We now examine the accuracy of the computed singular values and singular vectors. For our next experiment, we pick the second test matrix `Low-rank plus decay' and fix $k =50$ and $p = 10$. The conclusions we describe below for this matrix are consistent with other test matrices. The left panel in \cref{fig:comp_svs_angles} shows the absolute error in singular values computed using the different methods. We observe in each method that the error in the larger singular values is smaller. As before, we see that `TwoSided' and \cref{alg:randsvd} with $q=0$ have the worst performance. Although `GenEig' does far better than these two methods,  \cref{alg:randsvd} with $q=1$ produces the most accurate singular values. Next, we compare the accuracy of the singular vectors by means of the canonical angles. Suppose $\B{Z},\Zh \in \R^{n\times k}$ have $\B{W}$ orthonormal columns, then the cosine of the canonical angles can be computed by the singular values of $\Zh\t \B{WZ}$. See~\cite{knyazev2002principal} for definitions and details of the computation. The center panel of \cref{fig:comp_svs_angles} plots the canonical angles between the ``exact'' and approximate left singular vectors; similarly, the right panel plots the canonical angles between the ``exact'' and approximate right singular vectors. For the left singular vectors, `TwoSided' is comparable with `GSVD(0)', whereas `GenEig' is comparable with `GSVD(1)'. However, for the right singular vectors, `GSVD(1)' is the most accurate. The accuracy of the right singular vectors can be explained by the arguments made in \cite{saibaba2019randomized} (see discussion after Theorem 1). While the accuracy of low-rank representations using `GenEig' and `GSVD(1)' is comparable, there is a considerable difference in the accuracy of the singular values and the singular vectors. For this reason, we do not use the `TwoSided' in subsequent experiments. 
}

\subsubsection{Effect of condition number} \cref{sec:analysis}, and in particular \cref{thm:gapdep,thm:gapindep}, clearly identified the role of condition number $\kappa_2(\B{T})$  in the error analysis. We investigate the extent of the impact of a large condition number. For this purpose, we select the `Low-rank plus decay' matrix and the same $\B{S}$ matrix, described in \cref{sssec:test}. However, we construct four different matrices for $\B{T}$ using \textsc{MATLAB}'s \verb|gallery/randsvd| with condition numbers $\{10,10^4,10^7,10^{10}\}$. The relative error in the low-rank approximation is plotted as a function of the target rank $k$; the oversampling parameter is still $p=10. $ The results are plotting in \cref{fig:condT}. Two main observations can be drawn: first, the error appears to increase with increasing condition number, confirming the analysis, and second, subspace iteration with $q=1$ is sufficiently accurate to counteract the effects of ill-conditioning. We observe similar trends for other choices of $\B{A}$ but do not report them here.

\begin{figure}[!ht]
    \centering
    \includegraphics[scale=0.5]{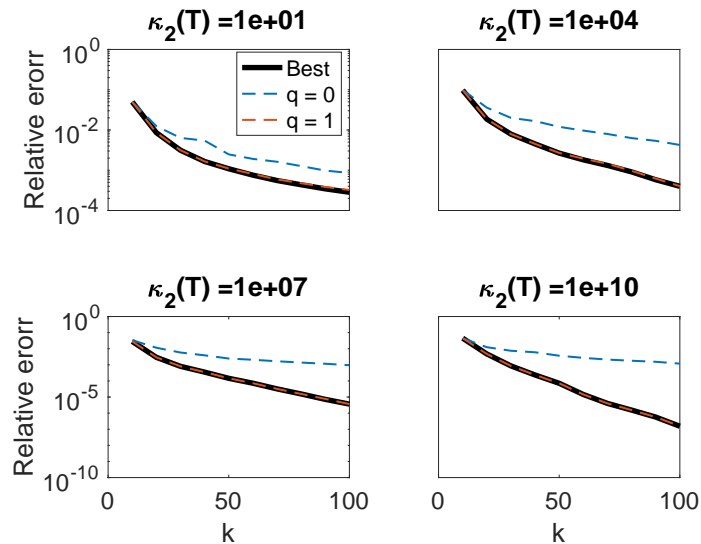}
    \caption{The relative error in the low-rank representations  obtained using subspace iteration with increasing condition number of $\B{T}$. In general, a larger condition number results in a larger error confirming the analysis in~\cref{ssec:prob}. Here, `Best' refers to the relative error in the best low-rank representation, i.e., $\sigma_{k+1}/\sigma_1$.}    \label{fig:condT}
\end{figure}
\begin{figure}[!ht]
    \centering
    \includegraphics[scale=0.4]{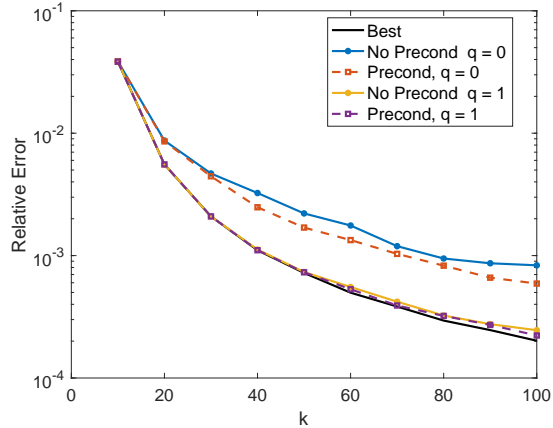}
    \caption{The effect of the preconditioning on the relative error in the low-rank representation. Here, `Best' refers to the relative error in the best low-rank representation, i.e., $\sigma_{k+1}/\sigma_1$.}
    \label{fig:precond}
\end{figure}
In the next experiment, we take $\B{T}$ as the matrix \verb|HB/no7| from the SuiteSparse matrix collection~\cite{davis2011university}. This matrix is of size $729\times 729$ with a condition number $\approx 2.4 \times 10^9$.   The matrices $\B{A}$ and $\B{S}$ are constructed in the same way as the previous experiment; the only difference is that are both of size $729\times 729$. The preconditioner is constructed using Incomplete Cholesky factorization via \textsc{MATLAB}'s \verb|ichol| with drop tolerance $10^{-4}$, and parameter `michol' set to `on'. The preconditioned operator has the condition number $\kappa_2(\B{LTL}\t) \approx 1.4 \times 10^5$. We compute the relative error in the low-rank representations with and without preconditioner and the results are reported in \cref{fig:precond}. It is readily seen that the use of the preconditioner improves the accuracy especially for the larger values of the target rank. However, as with the previous experiment, $q=1$ subspace iterations is sufficient to produce an accurate low-rank representations. We conclude if $\B{T}^{-1}$ is cheap to apply, then an additional round of subspace iteration is recommended. If, on the other hand, $\B{T}$ is too expensive to apply, but a good preconditioner is available, it can be used to improve accuracy.

\subsection{Numerical Study of Inexactness in HDSA} \label{ssub:hdsa_inexact}
In our next experiment, we discuss issues related to inexact matvecs in HDSA. Recall that $\B{A}$ arises from the discretization of the operator $\mathcal{D}z^\star = \Pi \mathcal K^{-1} \mathcal B$. The computation of a matrix-vector product $\B{Ax}$ can be done as 
\[ \B{Ax} = \B{\Pi}(\B{K}^{-1}(\B{Bx})),\]
where $\B\Pi,\B{K},\B{B}$ are discretized versions of $\Pi,\mathcal K$ and $\mathcal B$ respectively. The application of the transpose of $\B{A}$ can be performed similarly.  It is clear that the dominant cost involves solving linear systems involving $\B{K}$, when $\B{K}$ is large and it is natural to turn to iterative methods such as Conjugate Gradient. In this experiment, we investigate the impact of the tolerance, used as a stopping criterion in the iterative algorithm, on the accuracy of the generalized singular values obtained using the randomized algorithm. To this end, we use operators from the source inversion example in~\cite{hart_vbw_herzog_HDSA} (which were constructed explicitly at significant computational cost for this experiment) and use Conjugate Gradient for solving linear systems involving $\B{K}$. We choose three different tolerances, which controls the relative residual, for the stopping criterion $\text{tol} \in \{10^{-3},10^{-6},10^{-9}\}$. For comparison, we also use a direct solver. In all the experiments we have used an oversampling parameter $p=10$. We see that the absolute error in the singular values is high when the tolerance is $10^{-3}$; however, as the tolerance is decreased, we see that the error is almost indistinguishable from the direct solution. This is consistent with the theory and numerical experiments in \cite{golub2000inexact,ye2011inexact,simoncini2003theory} but more rigorous error analysis is needed to explore the effect of inexactness on the quality of the low-rank approximation. 

\begin{figure}[!ht]
    \centering
    \includegraphics[scale=0.5]{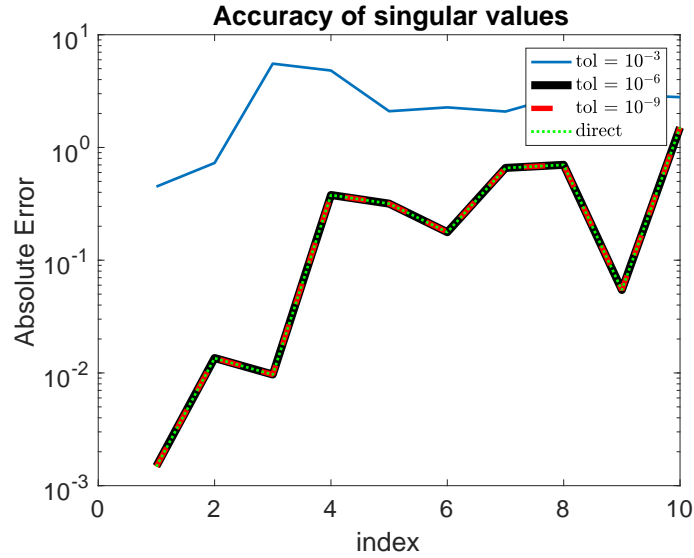}
    \caption{The effect of the tolerance used in the stopping criterion on the absolute error in the singular values. The label `direct' refers to using a direct solver for solving systems with $\B{K}$.  }
    \label{fig:hdsa_tol}
\end{figure}

\subsection{Application of HDSA for PDE-constrained optimal control} \label{ssub:hdsainv}
In this subsection we demonstrate the utility of the randomized GSVD to compute hyper-differential sensitivity indices for a PDE-constrained optimal control problem with spatially heterogeneous uncertain parameters in a model for subsurface fluid flow. Direct computation of such sensitivity indices, in this example, would require 8,800 matvecs, whereas by using the GSVD we reduce this to 96 matvecs. This results in a dramatic computational savings since each matvec requires many PDE solves.
\paragraph{Control Problem:} Consider the PDE-constrained optimal control problem 
\begin{align} 
& \min\limits_{u,z} \frac{1}{2} \sum\limits_{j=1}^r (\mathcal P_j u-T_j)^2 + \frac{\alpha}{2} \sum\limits_{i=1}^m z_i^2  \label{opt_control}  
\end{align}
such that
 \begin{align} -\nabla \cdot (\mu \nabla u) = & \> \sum\limits_{i=1}^m z_i f_i & \text{ in } \Omega \nonumber \nonumber \\
 u= & \>\psi & \text{on } \Gamma_D \nonumber \\
 -\mu \nabla u \cdot \B{n} = & \> 0 & \text{on } \Gamma_N \nonumber
\end{align}
where $\Omega=(0,1)^3$ is the computational domain with Dirichlet boundary $\Gamma_D = \{0\} \times (0,1)^2 \cup \{1\} \times (0,1)^2$ and Neumann boundary $\Gamma_N$ given by the remaining six sides of the unit cube $\Omega$, i.e. $\partial \Omega \setminus \Gamma_D$; $\B{n}$ denotes the outward pointing normal vector to the boundary. We use $(x_1,x_2,x_3) \in \Omega$ to denote spatial coordinates when needed. 

The PDE in \eqref{opt_control} arises from Darcy's law for subsurface porous media flow. It depends upon the spatially heterogeneous permeability field $\mu$ and a source term defined by a sum of $m$ spatially localized injections $f_i$, $i=1,2,\dots,m$, which are weighted by entries of the control vector $\B{z} \in \mathbb R^m$. The optimal control problem seeks to determine injection magnitudes $z_i$, $i=1,2,\dots,m$, so that the pressure field $u$ attains a target value $T_j$, $j=1,2,\dots,r$, at $r$ different spatial coordinates, $\mathcal P_j$ is the point-wise evaluation operator mapping $u$ to its value at the $j^{th}$ spatial coordinate.

Assuming that the permeability field $\mu$ and Dirichlet boundary condition $\psi$ are known, \cref{opt_control} may be solved using PDE-constrained optimization techniques. However, $\mu$ and $\psi$ are typically difficult to estimate in practice which motivates interest in the sensitivity of the optimal controller to perturbations of $\mu$ and $\psi$.

\cref{fig:opt_solution} displays the optimal controller and a comparison of the target pressure with the controlled and uncontrolled state. The left panel displays the optimal injection magnitudes $z_i$, $i=1,2,\dots,m$, with their spatial location given by the position of the circle in the cube $\Omega$ and the injection magnitude indicated by the circle's color. The injections are permitted in sixteen predetermined wells. The right panel compares the target data, uncontrolled state, and controlled state in a $x_1-x_2$ cross section with $x_3=0.53$ fixed. The vertical axis in this figure indicates the pressure. The sixteen cyan circles indicate the location (in the $x_1-x_2$ cross section) of the sixteen injection wells where the controller acts.
\begin{figure}
    \centering
    \includegraphics[width=0.49\textwidth]{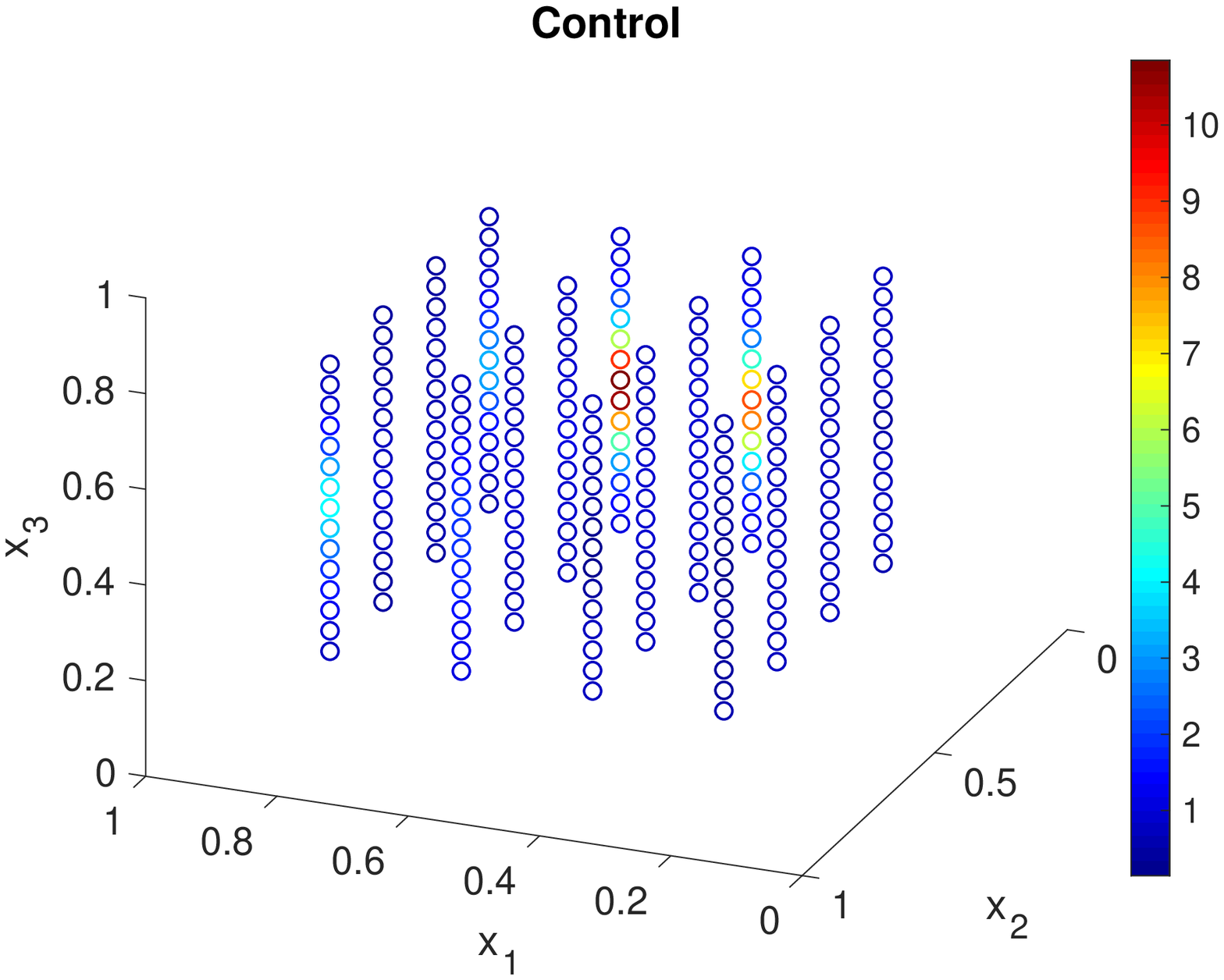}
    \includegraphics[width=0.49\textwidth]{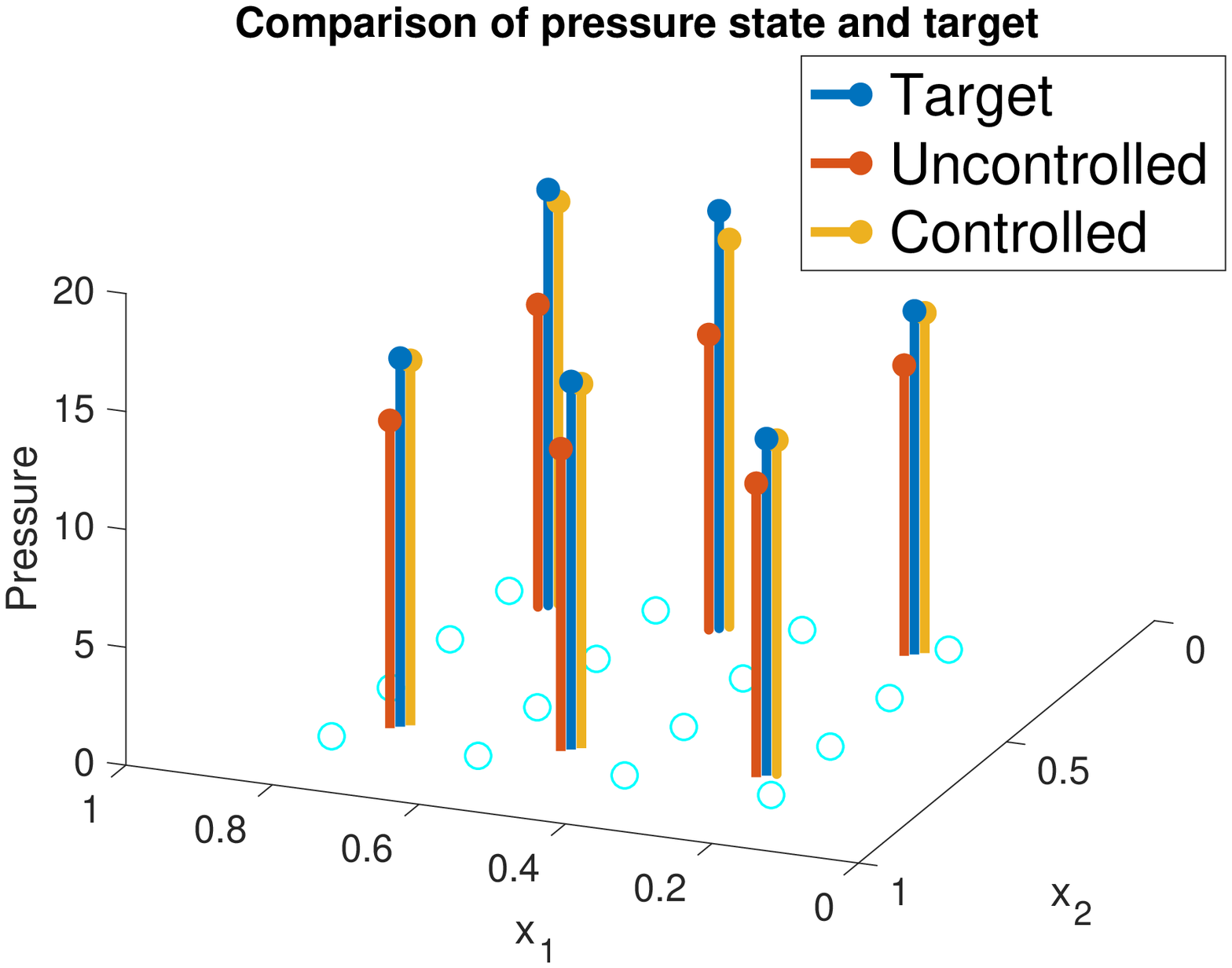}
    \caption{Left: optimal control strategy (injection magnitudes) in sixteen wells; right: comparison of the target data, uncontrolled, and controlled state, the cyan circles indicate the positions of the injection wells (coinciding with the sixteen wells in the left figure).}
    \label{fig:opt_solution}
\end{figure}

\paragraph{Parameter specification:} We represent uncertainty in the permeability field $\mu$, a function defined on $\Omega$, and Dirichlet boundary conditions $\psi_0$ (at $x_1=0$) and $\psi_1$ (at $x_1=1$), functions defined on $\Gamma_D$, by prescribing nominal estimates $\overline{\mu}$, $\overline{\psi_0}$, and $\overline{\psi_1}$, and considering
\begin{align*}
\mu(x_1,x_2,x_3) &= \overline{\mu}(x_1,x_2,x_3)\left(1 + a \sum\limits_{i=1}^{L+1} \sum\limits_{j=1}^{L+1}  \sum\limits_{k=1}^{L+1} \theta_{i,j,k}^{\mu} \phi_i(x_1)\phi_j(x_2)\phi_k(x_3) \right) \\
\psi_0(x_2,x_3) &= \overline{\psi_0}(x_2,x_3)\left(1 + a \sum\limits_{j=1}^{L+1}  \sum\limits_{k=1}^{L+1} \theta_{j,k}^{\psi_0} \phi_j(x_2)\phi_k(x_3) \right) \\
\psi_1(x_2,x_3) &= \overline{\psi_1}(x_2,x_3)\left(1 + a\sum\limits_{j=1}^{L+1}  \sum\limits_{k=1}^{L+1} \theta_{j,k}^{\psi_1} \phi_j(x_2)\phi_k(x_3) \right) \\
\end{align*}
where $a=0.1$ represents the level of uncertainty, $L=19$ is an integer specifying a spatial discretization, and $\phi$ is a linear finite element basis function defined on the interval $[0,1]$ with $L+1$ equally space nodes. Concatenating $\B{\theta^\mu}$, $\B{\theta^{\psi_0}}$, and $\B{\theta^{\psi_1}}$ yields a parameter vector $\B{\theta} \in \mathbb R^{8800}$. The discretization of $\mathcal D z^\star$ \eqref{eqn:dz}, with nominal parameter estimates, i.e. $\B{\theta}=0$, is a $240 \times 8800$ matrix for which we seeks its largest singular values and vectors to estimate sensitivity indices.

\subsubsection{Accuracy of the randomized GSVD algorithms}

The randomized GSVD of~\eqref{eqn:dz} is computed with an oversampling factor of $p=12$ and $q=1$ subspace iteration. The matrix $\B{S}$ arising from the mass matrices defined by inner products of the $f_i$'s and the matrix $\B{T}$ arises from the mass matrices defined by inner products of the $\phi_i$'s. The generalized singular values of~\eqref{eqn:dz} are given in \cref{fig:singular_values}. A low rank structure is identified, where sensitivities in the $8800$ dimensional discretized parameter space can be well approximated with small number of singular modes.

\begin{figure}[!ht]
    \centering
    \includegraphics[width=0.4\textwidth]{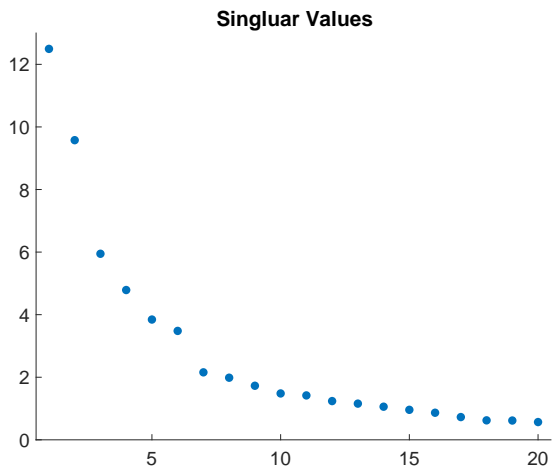}
    \caption{The 20 leading generalized singular values of $\mathcal D z^\star$.}
    \label{fig:singular_values}
\end{figure}

\subsubsection{Estimating sensitivity indices}

\cref{fig:permeability} displays the log of the nominal permeability field $\overline{\mu}$ (left) and the permeability field sensitivity indices (right). In the sensitivity indices panel (right), the size of the dot is scaled by the magnitude of the sensitivity index to aid in visualization of the high sensitivity regions. In particular, we notice a small region of red dots indicating that the high sensitivity is localized to a small region in space.

\begin{figure}[!ht]
    \centering
    \includegraphics[width=0.49\textwidth]{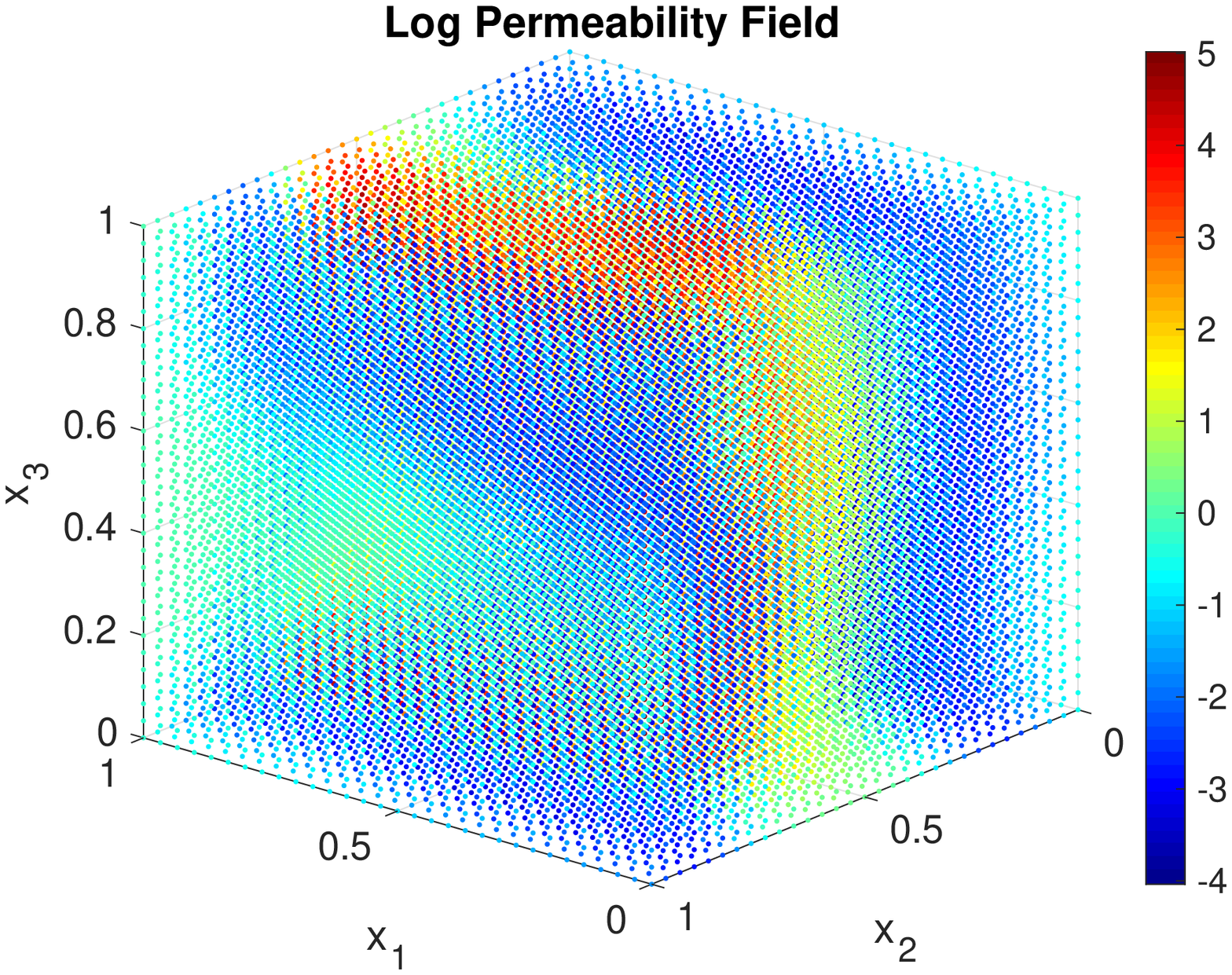}
    \includegraphics[width=0.49\textwidth]{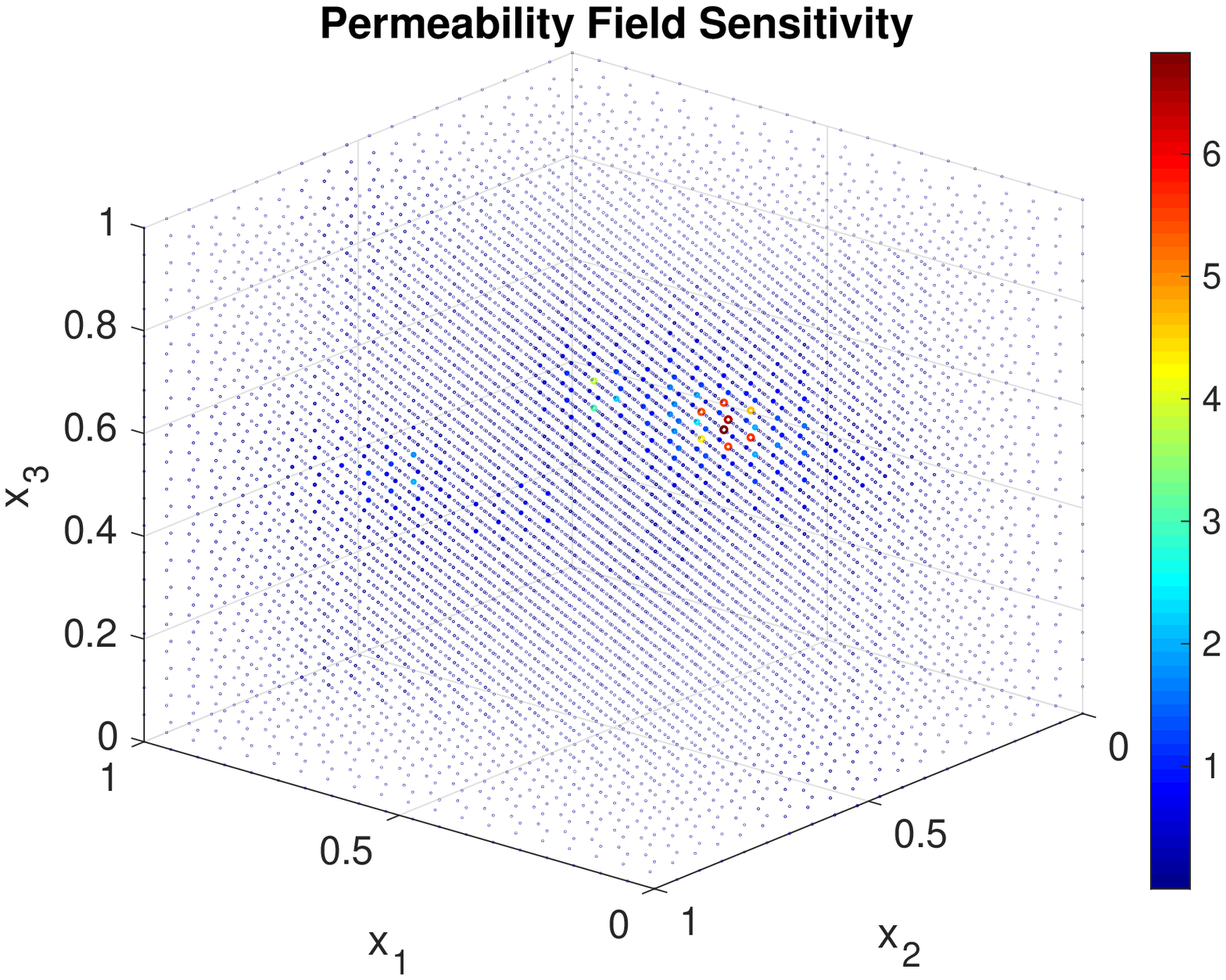}
    \caption{Left: natural log of the nominal permeability field $\overline{\mu}$; right: permeability field sensitivity indices.}
    \label{fig:permeability}
\end{figure}

\cref{fig:BC} displays the boundary condition sensitivity indices corresponding to the $x_1=0$ Dirichlet boundary on the left and $x_1=1$ Dirichlet boundary on the right. We observe spatially localized sensitivity in the boundary conditions. The sensitivity on the $x_1=1$ boundary is an order of magnitude larger than on the $x_1=0$ boundary, while both boundary sensitivities are smaller than the largest permeability field sensitivity indices.

\begin{figure}[!ht]
    \centering
    \includegraphics[width=0.49\textwidth]{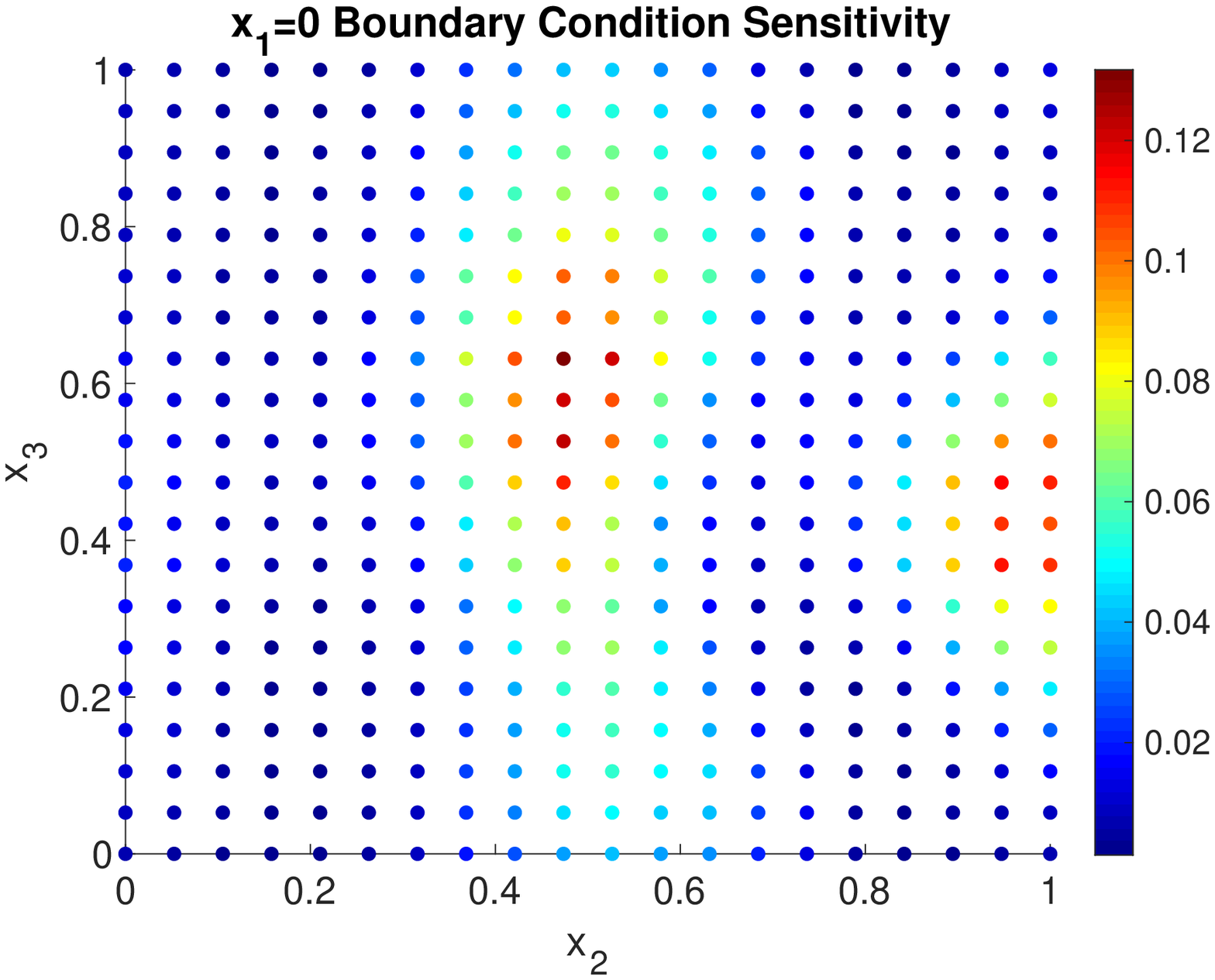}
    \includegraphics[width=0.49\textwidth]{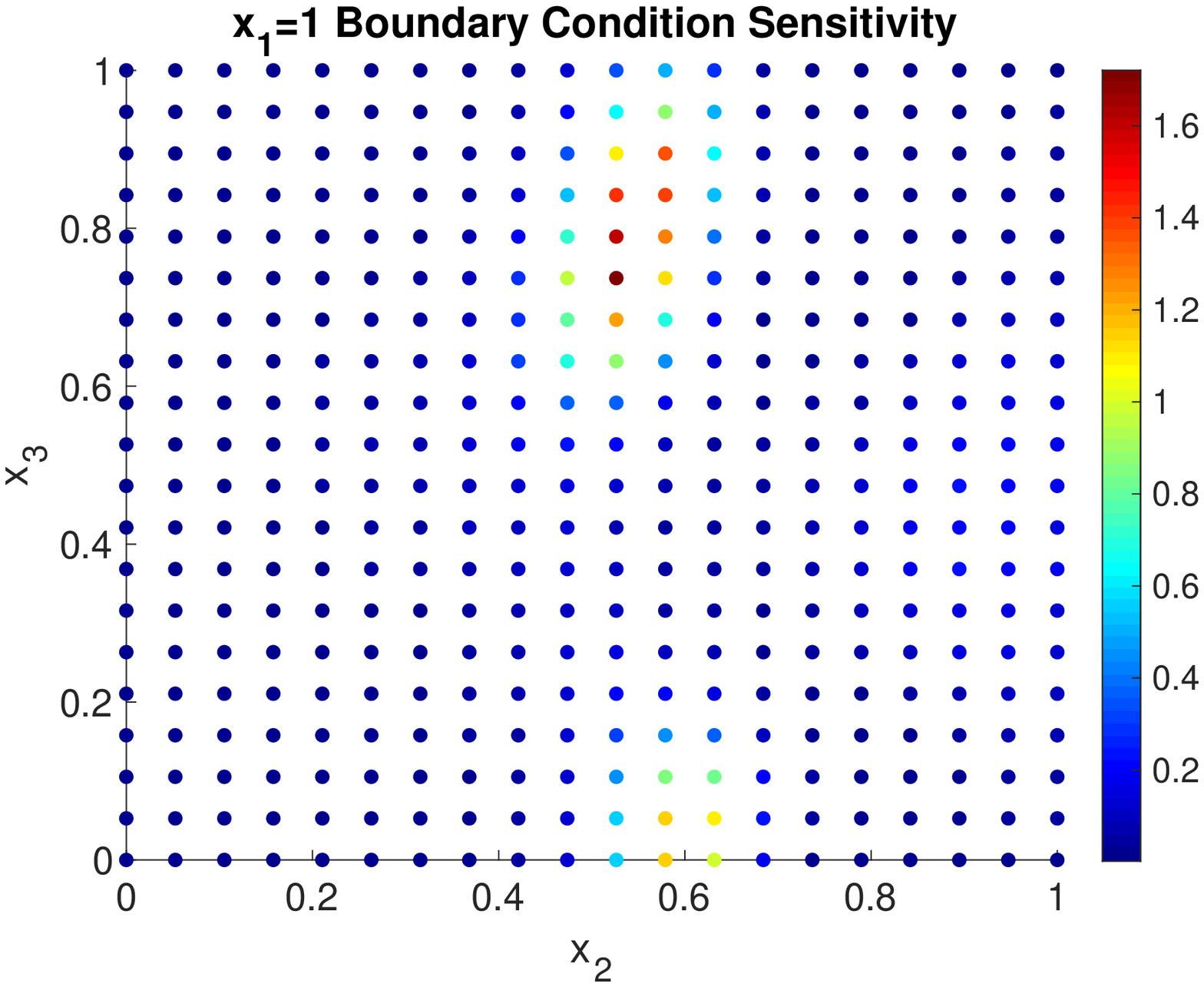}
    \caption{Left: $x_1=0$ Dirichlet boundary condition sensitivity indices; right: $x_1=1$ Dirichlet boundary condition sensitivity indices.}
    \label{fig:BC}
\end{figure}